\documentclass[11pt]{amsart}

\usepackage{a4wide, amsmath, amsfonts, amssymb, mathrsfs, amsthm, bbm, appendix}
\usepackage[hyperfootnotes=false]{hyperref}
\usepackage{xargs}
\usepackage[pdftex,dvipsnames]{xcolor}  % Coloured text etc.
\usepackage[shortlabels]{enumitem}

\usepackage{xcolor}
\numberwithin{equation}{section}
\usepackage{enumitem}

\newtheorem{theorem}{Theorem}[section]
\newtheorem{lemma}[theorem]{Lemma}
\newtheorem{corollary}[theorem]{Corollary}
\newtheorem{remark}[theorem]{Remark}
\newtheorem{proposition}[theorem]{Proposition}
\newtheorem{definition}[theorem]{Definition}
\newtheorem{example}[theorem]{Example}
\newtheorem{assumption}[theorem]{Assumption}

\allowdisplaybreaks[2]

\renewcommand{\d}{\mathrm{d}}

\renewcommand{\epsilon}{\varepsilon}
\newcommand{\R}{\mathbb{R}}
\newcommand{\N}{\mathbb{N}}
\renewcommand{\P}{\mathbb{P}}
\newcommand{\E}{\mathbb{E}}
\newcommand{\indicator}[1]{\mathbbm{1}_{#1}}
\newcommand{\Id}{\,\mathrm{d}}
\newcommand{\norm}[1]{\left\lVert#1\right\rVert}
\newcommand{\qv}[1]{\langle #1 \rangle}
\newcommand{\testfunctions}[1]{C_c^{\infty}(#1)}

\newcommand{\cF}{\mathcal{F}}
\newcommand{\cG}{\mathcal{G}}
\newcommand{\cP}{\mathcal{P}}

\usepackage{graphicx}
\usepackage{tikz}
\usepackage{atbegshi}

\AtBeginShipoutFirst{%
     \begin{tikzpicture}[remember picture, overlay]
         \node[anchor=north west, xshift=1in, yshift=-1.5cm] at (current page.north west) {%
             
\includegraphics[width=4cm]{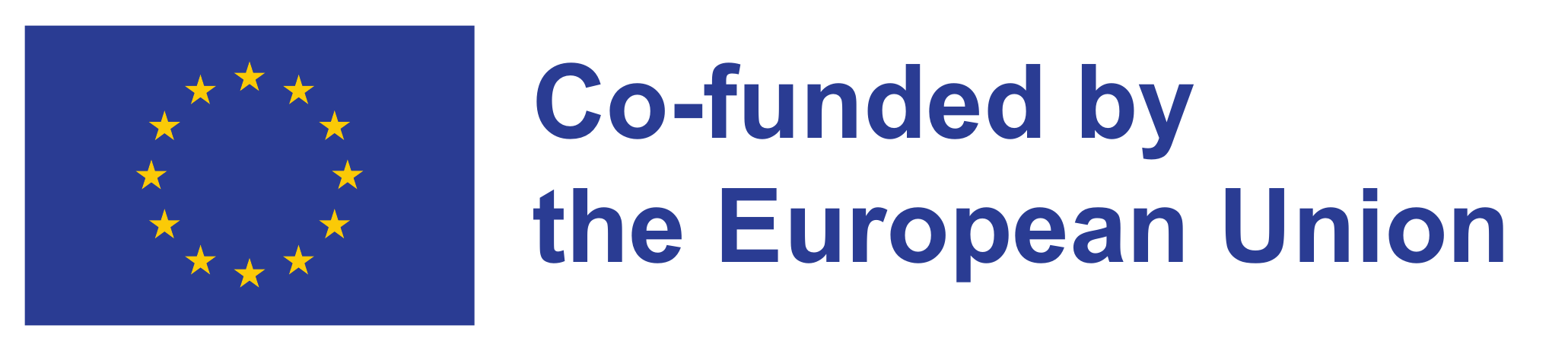} % Change filename, size or position as needed
         };
     \end{tikzpicture}%
}

\title[Fluctuation behaviour for interacting particle systems with common noise]{Fluctuation behaviour for interacting particle systems with common noise}

\author[Nikolaev]{Paul Nikolaev}
\address{Paul Nikolaev, University of Padova, Italy}
\email{paul.nikolaev@unipd.it}

\date{\today}

\begin{document}
\begin{abstract}We consider the asymptotic behaviour of the fluctuation process for large stochastic systems of interacting particles driven by both idiosyncratic and common noise with an interaction kernel \(k \in L^2(\R^d) \cap L^\infty(\R^d)\). Our analysis relies on uniform relative entropy estimates and Kolmogorov's compactness criterion to establish tightness and convergence of the fluctuation process. In this framework, an extension of the exponential law of large numbers is used to derive the necessary uniform estimates, while a conditional Fubini theorem is employed in the identification of the limit in the presence of common noise. We demonstrate that the fluctuation process converges in distribution to the unique solution of a linear stochastic evolution equation. This work extends previous fluctuation results beyond the classical Lipschitz framework.
\end{abstract}

\maketitle

\noindent \textbf{Key words:} common noise, interacting particle systems, Gaussian fluctuation, central limit theorem, propagation of chaos with common noise, relative entropy.

\noindent \textbf{MSC 2010 Classification:} 60H15, 60K35. 

\maketitle

\section{Introduction}
In this article, we consider interacting particle systems characterized by the systems of stochastic differential equations (SDE's) on the whole space \(\R^d\) 
\begin{equation} \label{eq: interacting_particle_system}
\Id X_t^{i} = - \frac{1}{N}\sum\limits_{j=1}^N k(X_t^{i}-X_t^{j}) \Id t + \sigma(t,X_t^{i}) \Id B_t^{i} + \nu(t,X_t^{i}) \Id W_t 
\end{equation}
driven by idiosyncratic noise \((B_t^{i}, t \ge 0)\), \(i \in \N\), and common noise \( (W_t, t \ge 0)\), both represented by multi-dimensional Brownian motions, some diffusion coefficients \(\sigma, \nu\) and some interaction kernel \(k\). The Brownian motions \((B_t^{i}, t \ge 0)\) are independent of each other, and \( (W_t, t\ge 0)\) is independent of \((B_t^{i}, t \ge 0)\) for all \(i\). Details of the probabilistic setting are provided in Section~\ref{sec: setting}.
Interacting particle systems pertubated by common noise like~\eqref{eq: interacting_particle_system} are commonly utilized in the field of mean-field games~\cite[Section 2.1]{CarmonaReneDElarue2018} as well as mathematical finance~\cite{Ahuja2016,Lacker2020,Hammersley2020,Shkolnikov2023}.
More precisely, many systems do not only experience idiosyncratic shocks but all player are also exposed to aggregated shocks (\(\nu \neq 0\)), which provides a more realistic view of social and biological
phenomena such as herding, bird flocking, fish schooling and interaction of agents. For instance, the understanding of common noise plays a crucial role in financial markets~\cite{Carmona2015} and mean-field games~\cite{Gueant2011,Gomes2014}.

Our aim is to investigate the asymptotic behavior of the system~\eqref{eq: interacting_particle_system} in the fluctuation scaling, which can be read as 
\begin{equation} \label{eq: fluctuation_scaling}
    \sqrt{N} \big( \mathrm{empirical \;  process \;  of}~\eqref{eq: interacting_particle_system} - \mathrm{mean} \textnormal{-} \mathrm{field \; density} )  . 
\end{equation}
Hence, we want to describe the deviations of the empirical measure 
\begin{equation}\label{eq: empirical_measure}
    \mu_t^N:= \frac{1}{N} \sum\limits_{i=1}^N \delta_{X_t^{i}}
\end{equation}
from the mean-field density, which satisfies the following stochastic Fokker--Planck equation 
\begin{align} \label{eq: chaotic_spde}
\begin{split}
\Id \rho_t
=& \;  \nabla \cdot (( k*\rho_t) \rho_t ))  \Id t 
- \nabla \cdot (\nu_t  \rho_t \Id W_t)  \\
&\; + \frac{1}{2}  \sum\limits_{i,j=1}^d \partial_{x_{i}}\partial_{x_{j}} \bigg( ( [\sigma_t\sigma_t^{\mathrm{T}}]_{(i,j)} + [\nu_t \nu_t^{\mathrm{T}}]_{(i,j)} ) \rho_t \bigg)  \Id t, 
\end{split}
\end{align}
where \([A]_{(i,j)}\) denotes the \(i,j\)-th entry of a matrix \(A\).  
The derivation of the mean-field density lies outside the scope of this article. We refer to the articles~\cite{Coghi2019,nikolaev2024common} for explicit well-posedness results.  
The main challenge in the presence of common noise is the stochastic nature of the limiting equation. While in the classical mean-field limit (\(\nu = 0\)), the limiting equation~\eqref{eq: chaotic_spde} is a deterministic parabolic partial differential equation, in our case, it remains stochastic. As a result, the lack of structure in the probability space introduces additional difficulties. These challenges were addressed in~\cite{CarmonaReneDElarue2018, CoghiFlandoli2016, nikolaev2024common, Shao2024}.

Hence, the main object of this article is the fluctuation process 
\((\eta_t^N, t \ge 0)\) given by 
\begin{equation} \label{eq: fluctuation_process}
    \eta_t^N := \sqrt{N} ( \mu_t^N -\rho_t)
\end{equation}
and its asymptotic behavior for \(N\to \infty\). 

In the case \(\nu = 0\), Wang, Zhao, and Zhu~\cite{Xianliang2023} studied the limit of \((\eta^N, N \in \N)\) on the torus \(\mathbb{T}^d\) with additive noise, i.e. \(\sigma = \mathrm{const.}\) and a bounded or symmetric kernel \(k\), which satisfies \(|\cdot | k(\cdot) \in L^1(\mathbb{T}^d)\), by utilizing relative entropy estimates for interacting particle systems, introduced by Jabin and Wang~\cite{JabinWangZhenfu2016, JabinWang2018}.

Our main Theorem extends the above result~\cite{Xianliang2023} to the unbounded domain \(\R^d\) and includes common noise \(W\) of multiplicative type. The classical methods in the common noise setting~\cite{Lacker2019} is no longer applicable, since the interaction kernel is no longer Lipschitz and the classical coupling method~\cite{snitzman_propagation_of_chaos} to produce quantitative estimates for the mean-field limit fails. 
Instead, we rely on the recently established relative entropy estimates in the common noise setting~\cite{nikolaev2024common}, which give rise to Assumption~\ref{ass: entropy} below. Furthermore, the presence of common noise alters the limit by disrupting the Gaussianity~\cite{Xianliang2023} of the limiting process, introducing additional challenges in its identification.

Let us summarize the main Theorem~\ref{theorem: main} of this article in a simplified version, which for the sake of presentation is stated
somewhat loosely. 
\begin{theorem}[Main Theorem (Informal Version)] \label{theorem: informal}
    Suppose that \((X_0^{i}, i \in \N)\) is i.i.d. with regular enough density \(\rho_0 \). Let \(k \in L^2(\R^d) \cap L^\infty(\R^d) \) The the sequence of measures \((\eta^N, N \in \N)\) of the interacting particle system~\eqref{eq: interacting_particle_system} converges in distribution in the space 
    \begin{equation*}
        L^2([0,T];H^{-\alpha}(\R^d)) \cap C([0,T]; H^{-\alpha-2}(\R^d))
    \end{equation*}
    for \( d/2 < \alpha \) towards the unique fluctuation SPDE~\eqref{eq: limiting_spde}. 
\end{theorem}

The restriction \(\alpha > d/2\) in the above theorem seems optimal, since this is precisely the range of \(\alpha\) for which probability measures can be embedded into Bessel potential spaces \(H^{-\alpha}(\R^d)\). Although, the regularity of the limiting process \(\rho\) can probably be improved by an argument similar to that used in~\cite[Section~3]{Xianliang2023}.

As already mentioned Theorem~\ref{theorem: informal} extends the results of~\cite{Xianliang2023} in two directions. Compared to~\cite{Xianliang2023}, we work on the whole space \(\R^d\) and consider multiplicative noise, as well as aggregated shocks introduced by the common noise term \(W\). This leads to new technical challenges that must be addressed. 

In particular, the identification of the limiting point of the sequence \((\eta^N, N \in \N)\) requires new arguments.

Moreover, the previous work implicitly relies on the fact that, in a bounded domain, we have \( L^p \subset \bigcap_{1 \leq q \leq p} L^q \), and therefore \( k \in L^\infty\) implies \(k\in L^1\), a property that crucially fails in unbounded domains. 
Moreover, several embedding theorems were used, which unfortunately are not available on \(\mathbb{R}^{d}\). One possible way to address this issue is to introduce weighted Sobolev spaces. On the one hand, the drawback is that negative Sobolev spaces are defined as the topological dual of the corresponding positive Sobolev space, so that the convenient representation via the Fourier transform or via Littlewood--Paley blocks is no longer available. On the other hand, weighted Sobolev spaces provide embedding theorems that allow us to characterize compact sets. We refer to~\cite{Jourdain1998} for further information on weighted Sobolev spaces and to~\cite[Inequalities~(0.3) and~(0.4)]{Jourdain1998} for the relevant embedding results.
Additionally, the compactness method, which is outlined in Section~\ref{sec: compactness_method} together with further comparison to the seminal paper~\cite{Xianliang2023}, relies on Kolmogorov's continuity theorem~\cite[Theorem 23.7]{Kallenberg2021}, as well as the crucial assumption concerning the tightness of the initial data \((\eta^N(0), N \in \mathbb{N})\). For more details we refer to Remark~\ref{remark: compactness_ness}. 
To overcome these challenges, we require a stronger version of the exponential law of large numbers~\cite[Lemma~2.3]{Xianliang2023}, which we present in Lemma~\ref{lemma: modification_super}.

Additionally, the approach developed here provides partial insight toward the Gaussian fluctuation result claimed in~\cite[Theorem~1.2]{Feng2023} for the two dimensional vortex model on the whole space \(\R^2\), corresponding to the case without common noise \(\nu = 0\). However, it appears that the result in~\cite[Theorem~1.2]{Feng2023} relies on additional assumptions that are not made explicit therein, and several of the technical challenges discussed above do not seem to be addressed.

Beyond these issues, our analysis indicates that a crucial step in the limiting procedure cannot be carried out within our framework. For transparency, we explicitly document which parts of our approach can be adapted to the vortex interaction model in Remark~\ref{remark: vortex_uniform},~\ref{remark: vortex_limit}. We also identify the point at which a fundamental obstruction arises in Remark~\ref{remark: vortex_stability}.
Hence, we believe this is still an open problem.

Another recent work on fluctuations, which we need to address is the result for the two-dimensional stochastic vortex model with common noise by Shao and Zhao~\cite{shao2025}. Similar to our approach, their work is a natural extension of their relative entropy estimates~\cite{Shao2024} for the stochastic vortex model, which plays a role analogous to our earlier work~\cite{nikolaev2024common}. Their analysis, like ours, is also primarily based on the techniques in~\cite{Xianliang2023}.

It is important to highlight the differences between their work and our current approach. As in the comparison between~\cite{Shao2024} and~\cite{nikolaev2024common}, key distinctions arise in the domain (\(\mathbb{T}^d\) vs. \(\mathbb{R}^d\)), the type of noise (transport noise vs. It\^o noise) and the regularity of the kernel \(k\). The challenges of working on an unbounded domain, particularly the use of \(k \in L^1\), have already been discussed in comparison to~\cite{Xianliang2023} and remain relevant here. Regarding the kernel, essential properties of the Biot--Savart law, such as its divergence-free nature and integrability on a bounded domain, play a crucial role~\cite{Shao2024,shao2025}. Moreover, while Shao and Zhao consider common noise, our work takes a more careful approach in Section~\ref{sec: identification_limit}, providing additional details on the augmentation of the filtration. In this regard, we refer to Lemma~\ref{lemma: brownian_motion} and the arguments therein.

Additionally, we completely avoid the reliance on \(L^2\)-convergence of the particle system to the McKean--Vlasov equation, as used in~\cite[Proposition~4.1]{shao2025}. At 
 the same time we eliminate the necessary of a strong solution to the McKean--Vlasov equation. Instead, we adopt the framework of~\cite{Lacker2019}, originally developed for smooth interaction kernels \(k\) and additive noise. This approach allows for a precise definition of the white noise term, which we identify in Lemma~\ref{lemma: characterization_martingale}. In particular, our demonstration of the conditional Gaussian property for the common noise term differs fundamentally, relying on distinct arguments.
Moreover, we provide Lemma~\ref{lemma: versions_of_stochastic_integrals} concerning the approximated common noise martingale term as a realization in the space \(H^{-\alpha}(\mathbb{R}^d)\), demonstrating the tightness of the common noise martingale in the Hilbert space \(H^{-\alpha}(\mathbb{R}^d)\).

\subsection{Methodology and difficulties} \label{sec: compactness_method}

In particular, the presence of common noise renders the limiting mean field equation random and disrupts the classical Gaussian structure of the fluctuation limit. As a consequence, the identification of the limiting object from the tightness of the fluctuation measures \((\eta^N, N \in \N)\) becomes substantially more delicate, since both the fluctuation process and the underlying mean field limit are random.
These issues necessitate new arguments in the Skorohod argument and limit identification steps, which cannot be obtained by a direct extension of the techniques in~\cite{Xianliang2023}. A more detailed analysis is postponed to Section

Our paper is closely related to the work of Wang, Zhao, and Zhu~\cite{Xianliang2023}. Both approaches rely on relative entropy estimates, originally established by Jabin and Wang~\cite{JabinWang2018} for the case \(\nu=0\) and later extended to the common noise setting by the author~\cite{nikolaev2024common}. The approach in~\cite{Xianliang2023} follows a classical compactness method, which has been employed in various works, including~\cite{Graham1996, Fernandez1997, Jourdain1998, Lacker2019}. Specifically, Wang, Zhao and Zhu utilize uniform entropy bounds on the relative entropy to demonstrate the tightness of the sequence of fluctuations \((\eta^N,N\in \N)\) in an appropriate Hilbert space.

 Let us sketch the general approach of the compactness method utilized in~\cite{Graham1996,Lacker2019,Xianliang2023} in order to demonstrate Gaussian fluctuations and where our approach varies. The same method is nicely summarized in~\cite{Hofmanova2012} for the weak existence of SDEs with bounded drift and diffusion. 
A crucial observation is that after an application of It\^{o}'s formula for \((\eta^N, N \in \N)\), we obtain  
\begin{align} \label{eq: etaN_spde}
    \Id \langle \eta_t^N, \varphi \rangle
    =&  \sqrt{N} \bigg( \langle \mu_t^N , \nabla \varphi \cdot (k*\mu_t^N) \rangle - \langle  \rho_t, \nabla \varphi  (k*\rho_t) \rangle  \bigg) \Id t 
    + \frac{1}{2} \langle \eta_t^N , \mathrm{Tr} \big( (\sigma \sigma^{\mathrm{T}} + \nu \nu^{\mathrm{T}} ) \nabla^2 \varphi \big) \rangle \Id t  \nonumber \\
    \quad &+ \langle \eta_t^N , \nu^{\mathrm{T}} \nabla \varphi \rangle \Id W_t + \frac{1}{\sqrt{N}} \sum\limits_{i=1}^N ( \sigma^{\mathrm{T}}_t(X_t^{i}) \nabla \varphi(X_t^{i}) ) \Id B_t^{i}
\end{align}
for some smooth function \(\varphi\). Hence, \((\eta_t^N, t \ge 0)\) solves informally the following SPDE
\begin{align} \label{eq: N_derivation}
\begin{split}
\Id \eta_t^N =& -  \nabla \cdot \Big(\eta_t (k*\rho_t^N) +  \rho_t (k*\eta_t^N) + \frac{1}{\sqrt{N}} \eta_t^N k*\eta_t^N \Big) \Id t    - \nabla \cdot (\eta^N_t  \nu^{\mathrm{T}} \Id W_t ) \\
& \quad
+  \frac{1}{2}  \sum\limits_{i,j=1}^d \partial_{z_{i}}\partial_{z_{j}} \bigg( ( [\sigma_t\sigma_t^{\mathrm{T}}]_{(i,j)} + [\nu_t \nu_t^{\mathrm{T}}]_{(i,j)} ) \eta_t^N \bigg)  \Id t 
 - \mathcal M_t^N ,
\end{split}
\end{align}
where \((\mathcal M_t^N, t \ge 0) \) is a martingale corresponding to the weighted sum of the stochastic integrals indexed by the Brownain motions \((B^{i}, i \in \N)\) in~\eqref{eq: etaN_spde}. 
Next, we deduce uniform estimates on \((\eta^N, N \in N)\) in an appropriated Hilbert spaces to utilize Kolmogorov's tightness criterion~\cite[Theorem~23.7]{Kallenberg2021} and Prokhorov's theorem. In this step, we utilize the idea introduced in~[13] of employing relative entropy estimates to derive uniform bounds and, consequently, tightness of the fluctuation sequence. In our setting, we also use relative entropy estimates in the presence of common noise, suitably adapted to the unbounded domain \(\R^d\), in order to obtain analogous uniform bounds. These arguments are carried out in Section~\ref{sec: uniform_estimates}.
The next step is to apply Skorohod's representation theorem, which allows one to show that the limit of \((\eta^N, N \in \N)\) satisfies the fluctuation limiting SPDE
\begin{align}\label{eq: limiting_spde}
\begin{split}
\Id \eta_t =& -  \nabla \cdot (\eta_t (k*\rho_t) ) - \nabla \cdot (\rho_t (k*\eta_t) )  \Id t - \nabla \cdot (\eta_t  \nu^{\mathrm{T}} \Id W_t ) \\
& \quad
+  \frac{1}{2}  \sum\limits_{i,j=1}^d \partial_{z_{i}}\partial_{z_{j}} \bigg( ( [\sigma_t\sigma_t^{\mathrm{T}}]_{(i,j)} + [\nu_t \nu_t^{\mathrm{T}}]_{(i,j)} ) \eta_t \bigg)  \Id t 
 - \nabla \cdot \big( \sigma^{\mathrm{T}} \sqrt{\rho_t} \xi ),
\end{split}
\end{align}
where \(\xi\) is the space time white noise. We will provide a precise definition of the SPDE~\eqref{eq: limiting_spde} later in Section~\ref{sec: setting}. In this step, we differ from~\cite{Xianliang2023} since we have additional stochasticity from the mean field limit. This leads to a technical challenges in the Skorohod argument. 

It is worth noting that the equation bears a resemblance to the highly singular (in the SPDE sense) Dean--Kawasaki equation~\cite{Dean96,Kawasaki94}. More precisely, if one replaces
\(\nabla \cdot \big( \sigma^{\mathrm{T}} \sqrt{\rho_t}\, \xi \big)\)
by
\(\nabla \cdot \big( \sigma^{\mathrm{T}} \sqrt{\eta_t}\, \xi \big)\),
one formally recovers the Dean--Kawasaki noise in our equation. Clearly, the Dean--Kawasaki equation describes a system with a fixed number of particles \(N\), whereas here we are interested in the limit \(N \to \infty\).

However, if we look at the \(N\)-particle equation~\eqref{eq: N_derivation} for constant diffusion coefficient \(\sigma\), we may formally replace the martingale term \(\mathcal M_t^N\) by the term \(\sigma \sqrt{\mu_t^N}\, \xi\), where \(\xi\) denotes space time white noise, since both terms share the same covariance structure. Linearizing around the mean-field limit \(\rho\), we obtain
\begin{equation*}
\sqrt{\mu_t^N}\, \xi = \sqrt{\rho_t}\, \xi + O(N^{-1/2}),
\end{equation*}
where the higher order terms vanish due to the quantitative mean-field limit. Hence, the fluctuation noise can be interpreted as the linearization of the Dean--Kawasaki noise around the mean-field limit, with higher order terms neglected due to the mean-field limit.

A crucial yet often overlooked step in Skorohod's representation theorem is the consideration of the new filtration after changing the probability space. Observe that the filtration must be complete and right-continuous to satisfy the usual assumptions, and all the processes must be adapted to the new filtration. This technical aspect is addressed, for instance, in Lemma~\ref{lemma: brownian_motion} and the subsequent results. 

We notice, that naively under a uniform bound on \((\eta^N, N \in \N)\) each term in~\eqref{eq: etaN_spde} should converge to the corresponding term in~\eqref{eq: limiting_spde} and 
\begin{equation}
    \mathcal M_t^N \to \nabla \cdot \big( \sigma^{\mathrm{T}} \sqrt{\rho_t} \xi ) , \quad \mathrm{as} \; N \to \infty. 
\end{equation}
Here, a major difference from the previous work~\cite{Xianliang2023} is that the right-hand side is no longer Gaussian prevents a direct adaptation of the identification step in~[13], and necessitates new techniques tailored to this setting. Instead, similar to the case studied in~\cite{Lacker2019}, the process is in a way only Gaussian when conditioned on the filtration \(\mathcal{F}^W\) generated by the common noise \(W\) (see Remark~\ref{remark: gaus}).
Hence, identifying the limit is a crucial and challenging task, which we address by employing a conditional covariation formula, as demonstrated in Lemma~\ref{lemma: quadratic_common_noise}.
In the final step, it remains to demonstrate that the SPDE~\eqref{eq: limiting_spde} is pathwise unique in the sense of SDEs. An application of Yamada--Watanabe's theorem completes the argument.

\subsection{Related literature}
In contrast to interacting particle systems driven solely by idiosyncratic noise~\cite{snitzman_propagation_of_chaos,JabinWang2018,RosenzweigSerfaty2023,galeati2024}, the literature on common noise remains relatively sparse. Early efforts on Gaussian fluctuations primarily focused on jump-type particle systems~\cite{McKean1975,Tanaka1982}, closely linked to the foundational work of~\cite{kac1956foundations}. In the special case where \(k=0,\; \nu=0,\; \sigma=\mathrm{const}\), It\^o demonstrated that the limit of the fluctuation process \((\eta^N_t, N \in \N)\) associated with the particle system~\eqref{eq: interacting_particle_system}, consisting of a sequence of independent one-dimensional Brownian motions, is itself Gaussian.

Subsequent work by  and M\'{e}l\'{e}ard extended~\cite{Fernandez1997} the study of fluctuations to the whole space for smooth interaction kernels \(k\), employing a Hilbert space framework and a compactness method outlined in the Introduction. Building on this approach, M\'{e}l\'{e}ard further investigated general McKean--Vlasov models~\cite{Graham1996} and, in collaboration with Jourdain~\cite{Jourdain1998}, explored moderate fluctuations, initially studied by Oelschl\"ager~\cite{Oeschlager1984,Oelschlager1990,Oelschlager1991,nikolaev2024}.
A closely related approach was later adopted by Wang, Zhao, and Zhu~\cite{Xianliang2023}, who derived Gaussian fluctuations on the torus, with a primary focus on the two-dimensional Biot--Savart kernel as a motivating example. As outlined in the Introduction, their key innovation lies in utilizing the relative entropy estimates from~\cite{JabinWang2018} to establish uniform bounds, which ultimately ensure the tightness of the fluctuation process \((\eta^N, N \in \N)\). 

There are numerous variations of the basic particle system~\eqref{eq: interacting_particle_system}. For instance, Lucon and Stannat~\cite{Stannat2016} studied its fluctuations under spatial constraints, while Grotto and Romito~\cite{Grotto2020} analyzed fluctuations in particle systems at stationarity. Moreover, Cecchin and Pelino~\cite{Cecchin2019} established a weak Gaussian fluctuation result for finite-state mean-field games. Additionally, fluctuations in the Coulomb gas, where the \(N\)-particle distribution follows a Gibbs measure, were investigated in~\cite{Serfaty2023}. For second-order systems, we refer to~\cite{Braun1977,Bernou2024} and the references therein.

Beyond studying convergence at individual time points \(t\), one can also examine fluctuations in the space \(C([0,T];\R^d)\), which is particularly relevant in the context of rough paths and pathwise propagation of chaos~\cite{Coghi2020,galeati2024}. An early contribution in this direction was made by Tanaka~\cite{Tanaka1984}, whose work remains a key inspiration for current research~\cite{galeati2024}.

Regarding common noise, Kurtz and Xiong demonstrated the convergence of the fluctuation process in a modified Schwartz space, where the common noise is modeled as white noise in the sense of Walsh~\cite{walsh1986}, and the coefficients satisfy at least Lipschitz regularity. As far as we know, apart from the later works~\cite{Lacker2019, shao2025}, this is the only result on fluctuations with common noise. Their approach combines martingale methods with a coupling technique. In our case, as well, the limiting process is no longer Gaussian but instead satisfies a stochastic evolution equation.

Extending these results to mean-field games, Delarue, Lacker, and Ramanan~\cite{Lacker2019} studied fluctuations in the presence of common noise under smooth coefficient assumptions. 
Due to the game nature of the interaction, their framework naturally requires strong regularity conditions on the coefficients. They utilize the seminal results on master equations to connect the particle system to an intermediate system through the master equation. In contrast, our particle system~\eqref{eq: interacting_particle_system} does not involve controls, eliminating the need for an intermediate system.

\smallskip
\noindent\textbf{Organization of the paper:}
In Section~\ref{sec: setting}, we provide the definitions of the particle systems and the associated SPDEs, along with an introduction to relative entropy and some preliminary results on the exponential law of large numbers~\cite{JabinWang2018}. In Section~\ref{sec: uniform_estimates}, we establish uniform estimates for the fluctuation process. Section~\ref{sec: tightness} uses the uniform bounds from Section~\ref{sec: uniform_estimates} to demonstrate the tightness of the fluctuation process~\eqref{eq: fluctuation_process}. Finally, in Section~\ref{sec: identification_limit}, we identify the limit as a solution of the SPDE~\eqref{eq: limiting_spde}. In combination with pathwise uniqueness, which we also establish in Section~\ref{sec: identification_limit}, we prove the Main Theorem~\ref{theorem: main}.

\section{Setting} \label{sec: setting}
We write a vector in \(\R^{dN}\) as \(x = (x_1, \ldots, x_N) \in \R^{dN} \), where \(x_i=(x_{i,1},\ldots,x_{i,d}) \in \R^d\). For a vector in \(\R^d\) we will use the variable \(z \in \R^d\). 
For a matrix \(A \in \R^{d \times d'}\) we denote the \((\alpha,\beta)\) entry as \([A]_{(\alpha,\beta)}\)
Throughout the entire paper, we use the generic constant \( C \) for inequalities, which may change from line to line at may depend on the dimension \(d\) and final time \(T\). For \( 1 \le p \le \infty\) we denote by \(L^p(\R^d)\) with norm \(\norm{\cdot}_{L^p(\R^d)}\) the vector space of measurable functions whose \(p\)-th power is Lebesgue integrable (with the standard modification for \(p = \infty\)), by \(\testfunctions{\R^d}\) the set of all infinitely differentiable functions with compact support on \(\R^d\) and by \(\mathcal{S}(\R^d)\) the set of all Schwartz functions.

Our main framework requires Bessel potential spaces. However, most of the Literature is formulated for Besov spaces \(B_{p,q}^s(\R^d)\), this is why we require to introduce the more general space.
We introduce the space of Schwartz distributions \(\mathcal{S}'(\R^d)\). We denote dual parings by \(\qv{\cdot, \cdot}\). For instance, for \(u \in \mathcal{S}', \; f \in \mathcal{S}\) we have \(\qv{u,f}  = u[f]\) and for a probability measure~\(\mu\) we have \(\qv{f,\mu}  = \int f \Id \mu\). The correct interpretation will be clear from the context but should not be confused with the scalar product \(\qv{\cdot, \cdot}_{L^2(\R^d)} \) in \(L^2(\R^d)\).

The Fourier transform \(\mathcal{F}[u]\) and the inverse Fourier transform \(\mathcal{F}^{-1}[u]\) for $u \in \mathcal{S}'(\R^d)$ and $ f \in \mathcal{S}(\R^d)$ are defined by
\begin{equation*}
  \qv{\mathcal{F}[u],f} := \qv{u, \mathcal{F}[f]},
\end{equation*}
where \(\mathcal{F}[f]\) and \(\mathcal{F}^{-1}[f]\) is given by
\begin{equation*}
  \mathcal{F}[f](\xi) := \frac{1}{(2\pi)^{d/2}} \int e^{-i \xi \cdot z } f(z) \Id z
  \quad \text{and}\quad
  \mathcal{F}^{-1}[f](\xi):= \frac{1}{(2\pi)^{d/2}} \int e^{i \xi \cdot z } f(z) \Id z .
\end{equation*} 

Let us now recall the Littlewood--Paley characterization of isotropic Besov spaces. A dyadic partition of unity $(\tilde \chi,\chi)$ in dimension $d$ is given by two smooth functions on $\R^d$ satisfying $\mathrm{supp}\;  \tilde \chi\subseteq \{x \in \R^d \, : \, |x| \le \frac{4}{3} \}$, $\mathrm{supp}\; \chi \subseteq\{x \in\R^d:\frac{3}{4}\le|x|\le\frac{8}{3}\}$ and $\tilde \chi(z)+\sum_{j \ge 0}\chi(2^{-j}z)=1$ for all $z\in\R^d$. We set
\begin{equation*}
  \chi_{-1}:=\tilde \chi\quad\text{and}\quad\chi_{j}:=\chi(2^{-j}\cdot)\quad\text{for }j\ge0.
\end{equation*}
Taking a dyadic partition of unity $(\tilde \chi,\chi)$ in dimension two, the \textit{Littlewood--Paley blocks} are defined as
\begin{equation*}
  \Delta_{-1}f:=\mathcal{F}^{-1}(\chi_{-1}\mathcal{F}f)\quad\text{and}\quad\Delta_{j}f:=\mathcal{F}^{-1}(\chi_{j}\mathcal{F}f)\quad\text{for }j\geq0.
\end{equation*}
Note that, by the Paley--Wiener--Schwartz theorem (see e.g. \cite[Theorem~1.2.1]{Triebel1978}), $\Delta_{j}f$ is a smooth function for every $j\geq-1$ and for every $f\in\mathcal{S}^{\prime}(\R^{2})$ we have 
\begin{equation*}
  f=\sum_{j\geq-1}\Delta_{j}f:=\lim_{j\to\infty}S_{j}f\quad\text{with}\quad S_{j}f:=\sum_{i\leq j-1}\Delta_{i}f.
\end{equation*}
For $s \in\mathbb{R}$ and $p,q\in(0,\infty)$ the Besov space \(B_{p,q}^s(\R^d)\) is defined as
\begin{align*}
    B_{p,q}^s(\R^d):&= \bigg\{ f \vert f \in \mathcal S'(\R^d), \norm{f}_{B_{p,q}^s(\R^d)} := \bigg(\sum\limits_{j=-1}^\infty 2^{sjq} \norm{\Delta_j f}_{L^p(\R^d)}^q \bigg)^{\frac{1}{q}}  < \infty \bigg\}. 
\end{align*}
We also use the natural modification for \(p,q=\infty\) for the Besov spaces \(B_{p,q}^s(\R^d)\). 

For each \(s \in \R\) we denoted the Bessel potential by \(J^s := (1-\Delta)^{s/2}f := \mathcal{F}^{-1}[(1+|\xi|^2)^{s/2} \mathcal{F}[f]] \) for \(f \in \mathcal{S}'(\R^d)\). We define the Bessel potential space \(\mathnormal{H}_p^s(\R^d) \) for \(p \in [1,\infty)\) and \(s \in \R\) by
\begin{equation*}
  H_p^s(\R^d):= \{ f \in \mathcal{S}'(\R^d) \; : \;  (1-\Delta)^{s/2} f \in L^p(\R^d) \}
\end{equation*}
with the norm
\begin{equation*}
  \norm{f}_{H^{s}_p(\R^d)}:= \norm{(1-\Delta)^{s/2}f}_{L^p(\R^d)}, \quad f \in  H_p^s(\R^d).
\end{equation*}
We make the following notational convention \(H^s(\R^d) :=H^s_2(\R^d) \). We remark that \(H^s(\R^d) = B_{2,2}^s(\R^d)=  F_{2,2}^s(\R^d)\)~\cite{Triebel1978}. 

\subsection{Probabilistic setting} \label{sec: prob_setting}
In this subsection we introduce the probabilistic setting, in particular, the \(N\)-particle system and the associated McKean--Vlasov equation. To that end, let \((\Omega, \mathcal{F}, ( \mathcal{F}_t, t \ge 0) , \P)\) be a complete probability space with right-continuous filtration \((\mathcal{F}_t, t \ge 0) \) supporting the following probabilistic objects. Let \((B_t^{i}=(B_t^{i,1},\ldots,B_t^{i,m}), t\ge 0)\), \( i=1, \ldots, N\) be independent \(m\)-dimensional Brownian motions with respect to \((\cF_t, t \ge 0)\) and \((W_t=(W_t^{1},\ldots,W_t^{\tilde{m}}), t\ge 0)\) be another \(\tilde{m}\)-dimensional Brownian motions with respect to \((\cF_t, t \ge 0)\), which is independent from \((B_t^{i}, t\ge 0)\), \( i=1, \ldots, N\). Moreover, we denote by \(\cF^W = (\cF^W_t ,t \ge 0)\) 
the augmented filtration generated by \(W\) and by
\(\cP^W\) the predictable \(\sigma\)-algebra with respect to \(\cF^W\). For the initial data we consider a sequence \((\zeta^{i}, i \in \N)\) of independent \(d\)-dimensional \(\cF_0\)-measurable random variables with density \(\rho_0\), which are independent of the Brownian motions \((B_t^{i}, t\ge 0)\), \( i=1, \ldots, N\) and the filtration \(\cF^W = (\cF^W_t ,t \ge 0)\). 
We require also the following coefficients 
\begin{equation*}
    k \colon \R^d \mapsto \R^d, \quad 
    \sigma \colon [0,T] \times \R^d \mapsto \R^{d \times m}, \quad 
    \nu \colon [0,T] \times \R^d \mapsto \R^{d \times \Tilde{m}}. 
\end{equation*}

Let \(Z\) be a Polish space. 
For a filtration \(\cG:=(\cG_t)_{t\ge 0}\), \(1 \le p \le \infty\) and \(0 \le s <t \le T\) we denote by \(L^p_{\cG}([s,t];Z)\) the set of \(Z\)-valued predictable processes \((X_u, u \in [s,t])\) with respect to \(\cG\) such that
\begin{align*}
  \norm{X}_{L^p_{\cG}([s,t];Z )}:=
  \begin{cases}
  \bigg( \E \bigg( \int\limits_{s}^t  \norm{X_u}_Z^p \Id u \bigg) \bigg)^{\frac{1}{p}}, \quad & p \in [1, \infty),   \\
  \sup\limits_{ (\omega,u)  \in \Omega \times [s,t]} \norm{X_u}_Z , \quad & p = \infty, 
  \end{cases}
\end{align*}
is finite. If the process is independent of the probability space, i.e. a function \([0,T] \to Z \), then we omit the filtration in the notation. 

Let \(\mu,\nu\) be two probability measures over \(Z\).
The relative entropy \(\mathcal H(\mu|\nu)\) is defined by
\begin{equation} \label{eq: intro_relative_entropy}
\mathcal{H}(\mu|\nu):=
\left\{
\begin{aligned}
&\int_E \frac{\d\mu}{\d \nu}\log\Big(\frac{\d \mu}{\d \nu}\Big)\d \nu,&\mu\ll\nu,\\
&\infty,&\mbox{otherwise,}
\end{aligned}
\right.
\end{equation}
where \(\tfrac{\d\mu}{\d \nu}\) denotes the Radon--Nikodym derivative. 

In order to compare the relative entropy in \(\R^{dN}\) we require the \(N\)-fold version of equation~\eqref{eq: chaotic_spde} as well as the conditional density of the interactive particle system~\eqref{eq: interacting_particle_system}. For each \(N\in \N\) let \((\rho_t^N,t \ge 0)\) denote the density of the particle system~\eqref{eq: interacting_particle_system} conditioned on \(\cF^W_t\). Then, as in~\cite{nikolaev2024common}, \((\rho_t^N,t \ge 0)\) solves the SPDE 
\begin{align} \label{eq: N_SPDE}
    \Id \rho_t^N
=& \;  \sum\limits_{i=1}^N \nabla_{x_i} \cdot \bigg( \frac{1}{N} \sum\limits_{j=1}^N k(x_i-x_j) \rho_t^N \bigg) \Id t 
- \sum\limits_{i=1}^N \nabla_{x_i} \cdot (\nu(t,x_{i})  \rho_t^N \Id W_t)  \nonumber \\
&\; + \frac{1}{2} \sum\limits_{i,j=1}^N \sum\limits_{\alpha,\beta=1}^d \partial_{x_{i,\alpha}}\partial_{x_{j,\beta}} \bigg( ( [\sigma(t,x_i)\sigma(t,x_j)^{\mathrm{T}}]_{(\alpha,\beta)} \delta_{i,j} + [\nu(t,x_i)\nu(t,x_j)^{\mathrm{T}}]_{(\alpha,\beta)} ) \rho_t^N  \bigg)  \Id t .
\end{align}
Similar, the \(N\)-fold product of \(\rho\) solves 

\begin{align} \label{eq: tensor_spde} 
\Id \rho_t^{\otimes N }
=& \;  \sum\limits_{i=1}^N \nabla_{x_i} \cdot (( k*\rho_t)(x_i) \rho_t^{\otimes N } ))  \Id t 
- \sum\limits_{i=1}^N \nabla_{x_i} \cdot (\nu(t,x_{i})  \rho_t^{\otimes N }  \Id W_t) \nonumber  \\
& + \frac{1}{2} \sum\limits_{i,j=1}^N \sum\limits_{\alpha,\beta=1}^d \partial_{x_{i,\alpha}}\partial_{x_{j,\beta}} \bigg( ( [\sigma(t,x_i)\sigma(t,x_j)^{\mathrm{T}}]_{(\alpha,\beta)} \delta_{i,j} + [\nu(t,x_i)\nu(t,x_j)^{\mathrm{T}}]_{(\alpha,\beta)} ) \rho_t^{\otimes N }  \bigg)  \Id t. 
\end{align}

In the following \(a\) will always denote a fixed constant, which is strictly greater than \(d/2+2\), i.e. 
\begin{equation*}
    a > \frac{d}{2}+2. 
\end{equation*}
Moreover, the variable \(\alpha\), will vary in the open interval \((\frac{d}{2}+2, a)\).

Next, we define the solution of SPDE~\eqref{eq: chaotic_spde}.
\begin{definition} \label{def: solution_spde}
    Let \(\alpha \in \R\). Given some probability space \((\Omega, \cF,\P)\) supporting a Brownian motion \((W_t, t \ge 0)\), we call a non-negative stochastic process \((\rho_t,t \ge 0)\) a solution to the SPDE~\eqref{eq: chaotic_spde} with values in the space \(H^{\alpha}(\R^d)\) with initial data \(\rho_0\), if 
    \begin{equation}
        \rho \in  L^2_{\cF^W}([0,T];H^{\alpha}(\R^{d}) ) 
    \end{equation}
    and, for any \(\varphi \in \testfunctions{\R^{d}}\), \(\rho\) satisfies almost surely the equation, for all \(t \in [0,T]\),
    \begin{align*} 
    \begin{split}
     \qv{\rho_t, \varphi}_{L^2(\R^{d})} 
    =&\;   \qv{\rho_0, \varphi}_{L^2(\R^{d})}
    -  \int\limits_0^t   \bigg \langle (k*\rho_s) \rho_s, \nabla_{x_i} \varphi \bigg \rangle \Id s  \\
   &    + \frac{1}{2} \sum\limits_{\alpha,\beta=1}^d  \int\limits_0^t \bigg\langle   [\sigma(s,z)\sigma(s,z)^{\mathrm{T}}]_{(\alpha,\beta)} \rho_s, \partial_{z_{\beta}}  \partial_{z_{\alpha}} \varphi \bigg\rangle \Id s \\
    &+  \frac{1}{2} \sum\limits_{\alpha,\beta=1}^d  \int\limits_0^t \bigg\langle   [\nu(s,z)\nu(s,z)^{\mathrm{T}}]_{(\alpha,\beta)} \rho_s, \partial_{z_{\beta}}  \partial_{z_{\alpha}} \varphi \bigg\rangle \Id s \\
    &    +  \sum\limits_{\alpha=1}^d \sum\limits_{\hat{l}=1}^{\tilde{m}} \int\limits_0^t \bigg\langle    \nu^{\alpha,\hat{l}}(s,z) \rho_s ,  \partial_{z_{\alpha}} \varphi \bigg\rangle \Id W_s^{\hat{l}}. 
    \end{split}
  \end{align*}
\end{definition}

\begin{definition} \label{def: limiting_spde}
We call \(\eta\) a martingale solution to the SPDE~\eqref{eq: limiting_spde} on some stochastic basis \((\Omega, \cF, (\cF_t, t \ge 0), \P)\) if,
\begin{enumerate}
\item The stochastic basis supports four processes \((\eta_t, t \ge 0)\), \((\rho_t, t \ge 0)\), \((W_t, t \ge 0)\) and \((\mathcal M_t,, t \ge 0)\). 
\item The process \((\eta_t, t \ge 0)\) is a continuous \(H^{-\alpha-2}(\R^d)\)-valued, \((\cF_t,t \ge 0)\) adapted process satisfying \(\eta \in L^2_{\cF}([0,T]; H^{-\alpha}(\R^d))\) for every \(a > \alpha > d/2\). 
    \item \(W\) is a \(\tilde m\) dimensional Brownian motion with respect to the filtration \((\cF_t, t \ge 0)\).  
    \item \((\rho_t,t \ge 0)\) is a solution of the SPDE~\eqref{eq: chaotic_spde} in \(H^a(\R^d)\). 
    \item The process \(\mathcal M\) conditioned on the filtration \((\cF_t^W, t \ge 0)\), which is generated by the process \((W_t, t \ge 0)\) is a continuous \((\cF_t, t \ge 0)\)-martingale with values in \(\bigcap\limits_{m \in \N} H^{-\frac{d}{2}-2-\frac{1}{m}}(\R^d)\) 
and its covariance is given by 
\begin{equation*}
    \E( \mathcal M_{t}(\varphi_1) \mathcal M_s(\varphi_2) \vert \cF_t^W  ) = \sum\limits_{l=1}^{m} \sum\limits_{q , \hat q =1}^d  \int\limits_0^{\min(s,t)} \langle \sigma_u^{q,l} \sigma_u^{\hat q, l} \partial_{x_q}\varphi \partial_{x_{\hat q }} \phi,  \rho_u \rangle  \Id u 
\end{equation*}
for each \(\varphi_1, \varphi_2 \in \testfunctions{\R^d}\). More precisely, let \(\varphi \in \testfunctions{\R^d}\). Then, for any real-valued bounded continuous function \(\gamma_1\) on \(C([0,s],\R^{\tilde m}) \) and real-valued bounded continuous function \(\gamma_2\) on \(\bigcap\limits_{m \in \N}  C([0,T],H^{-\frac{d}{2}-2-\frac{1}{m}}(\R^d)) \) we have 
\begin{equation*}
    \E( \gamma_1 (W) \gamma_2(\mathcal M_{|[0,s]}) (\mathcal M_t(\varphi) -\mathcal M_s(\varphi)) = 0.  
\end{equation*}
    
    \item For each \(\varphi \in \testfunctions{\R^d}\) and \(t \in [0,T]\) it holds \(\P\)-a.e. that 
    \begin{align}\label{eq: def_spde_weak_formulation}
\begin{split}
     \langle \eta_t, \varphi \rangle
    =& \langle \eta_0, \varphi \rangle+   \int\limits_0^t \langle \rho_s (k*\eta_s)  , \nabla \varphi \rangle + \langle  \eta_s (k*\rho_s) , \nabla \varphi  \rangle     \Id s
    +\int\limits_0^t  \langle \eta_s, \nu^{\mathrm{T}}_s \nabla \varphi \rangle \Id W_s  \\
    \quad & +  \frac{1}{2}\int\limits_0^t  \langle \eta_s , \mathrm{Tr} \big( (\sigma_s \sigma^{\mathrm{T}}_s + \nu_s \nu^{\mathrm{T}}_s ) \nabla^2 \varphi \big) \rangle \Id s + \mathcal M_t(\varphi), 
    \end{split}
\end{align}
\end{enumerate}

\end{definition}

\begin{remark} \label{remark: gaus}
The following observation concerning the process \(\mathcal M\) in Definition~\ref{def: solution_spde} can be made.
In the case without common noise, that is \(\nu = 0\), it is well known from the central limit theorem for interacting particle systems together with the martingale central limit theorem that the limiting fluctuation process \(\mathcal M\) is Gaussian. We refer for instance to~\cite{Fernandez1997,Xianliang2023}.

When \(\nu \neq 0\), the situation changes substantially. In this case, the quadratic variation and hence the covariance structure of \(\mathcal M\) depend on the solution \(\rho\) of the limiting SPDE, which is itself random due to the presence of the common noise. As a consequence, the process \(\mathcal M\) is no longer Gaussian in an unconditional sense.

However, if the realization of the common noise \(W\) is fixed, then the path of \(\rho\) is determined, and therefore the covariance structure of \(\mathcal M\) becomes deterministic. In this conditional framework, the martingale central limit theorem implies that \(\mathcal M\) is a Gaussian process conditionally on the common noise \(W\). For this reason, we say that \(\mathcal M\) is Gaussian conditionally on \(W\).

We refer to~\cite{Lacker2019} for a detailed analysis of this phenomenon in the closely related setting of mean field games with common noise and smooth coefficients. From this perspective, our main result, Theorem~\ref{theorem: informal}, shows that in the fluctuation limit the particles exhibit Gaussian behavior once the random environment generated by the common noise is fixed. This is consistent with the classical picture in the absence of common noise, as discussed for example in~\cite{Xianliang2023}.

Finally, we emphasize that it is the solution \(\eta\) of the linear fluctuation SPDE admits the structure of an infinite dimensional Ornstein--Uhlenbeck process. Indeed, \(\eta\) solves a linear SPDE with driving noise given by the conditional martingale \(\mathcal M\). In the case \(\nu = 0\), the coefficients of this linear equation are deterministic, and \(\eta\) is an infinite dimensional Ornstein--Uhlenbeck process in the classical sense, and in particular a Gaussian process. When \(\nu \neq 0\), the coefficients of the linear SPDE depend on the random solution \(\rho\) of the limiting SPDE, and therefore the Ornstein--Uhlenbeck dynamics evolves in a random environment. In this case, \(\eta\) is no longer Gaussian unconditionally, but it is Gaussian conditionally on the common noise \(W\). 

\end{remark}

\begin{definition}[Pathwise uniqueness for SPDE~\eqref{eq: chaotic_spde}]
    Let \(\alpha \in \R\). We say pathwise uniqueness holds in \(H^\alpha(\R^d)\) for the SPDE~\eqref{eq: chaotic_spde}, if for any two solutions \(\rho^1,\rho^2\) on the same probability space \((\Omega, \cF, \P)\) with the same Brwonian motion \((W_t, t \ge 0)\) the properties in Definition~\ref{def: solution_spde} hold for \(\rho^1,\rho^2\) and 
    \begin{equation*}
        \P \Big(\sup\limits_{0 \le t \le T} \norm{\rho^1_t-\rho^2_t}_{H^{\alpha}(\R^d)} = 0\Big)=1. 
    \end{equation*}
\end{definition}

\begin{definition}[Pathwise Uniqueness for SPDE~\eqref{eq: limiting_spde}] \label{def: pathwise_unique}
We say that pathwise uniqueness holds for the SPDE \eqref{eq: limiting_spde}, if for any two martingale solutions \((\eta_t^1, t \ge 0)\) and \((\eta_t^2, t \ge 0)\) on the same stochastic basis \((\Omega, \mathcal{F}, \mathbb{P})\), with the same Brownian motions \((W_t, t \ge 0))\), \((M_t, t \ge 0 )\), and the same initial data \(\eta_0 \in \bigcap_{m \in \mathbb{N}} H^{-d/2 -2- \frac{1}{m}}(\R^d)\), and same solution of the stochastic Fokker--Planck \((\rho_t, t \ge 0)\), it holds that for every \(a>  \alpha > d/2+2\),
\[
\mathbb{P} \Big( \sup_{t \in [0,T]} \|\eta_t^1 - \eta_t^2 \|_{H^{-\alpha}(\R^d)} = 0 \Big) = 1.
\]
\end{definition}

\subsection{Main result and assumptions}
Throughout the article we require the following set of assumptions on the interaction kernel and diffusion coefficients. 
\begin{assumption}[Assumptions on the coefficients] \label{ass: main} \
    \begin{enumerate}
        \item The diffusion coefficients \(\sigma , \nu\) lie uniformly in the space \( B^{a}_{\infty, \infty}\) for \(a > d+2\), i.e,
        \begin{equation*}
            \max\limits_{i,l,\tilde l} \sup\limits_{0 \le t \le T} \bigg( \norm{\sigma^{i,l}(t,\cdot)}_{B_{\infty,\infty}^a(\R^d)}+ \norm{\nu^{i,\tilde l}(t,\cdot)}_{B_{\infty,\infty}^a(\R^d)}\bigg) \le C 
        \end{equation*}
        for some positive constant \(C>0\). 
        \item \label{item: ellipticity_idiosy} \(\sigma\) satisfies the ellipticity condition. For all \(\lambda \in \R^d\) we have 
\begin{equation*} 
\sum\limits_{\alpha,\beta=1}^d      [\sigma_s(z)\sigma_s(z)^{\mathrm{T}}]_{(\alpha,\beta)}    \lambda_{\alpha} \lambda_{\beta} \ge \delta |\lambda|^2 . 
\end{equation*}
        \item The interaction kernel satisfies \(k \in L^2(\R^d) \cap L^\infty(\R^d)\).  
    \end{enumerate}
\end{assumption}

 Under the above assumption and using Zvonkin's transformation, we obtain strong existence and uniqueness for the interacting particle system~\eqref{eq: interacting_particle_system}. For instance, see~\cite{hao2022} in the case \(\nu = 0\). 

Regarding Assumption~\ref{ass: main}~(i), it is well known that the Besov space \(B^{s}_{\infty,\infty}(\R^d)\) coincides with the Zygmund space \(\mathcal{C}^s(\R^d)\) for \(s > 0\) (see, e.g.,~\cite[Theorem~2.5.7]{Triebel1978}). Consequently, a stronger alternative to Assumption~\ref{ass: main}~(i) would be for instance to require that \(\sigma\) and \(\nu\) have bounded and continuous derivatives up to order \(\lfloor a \rfloor + 1\).

Additionally, we need assumptions for the initial data and the mean-field limits SPDE~\eqref{eq: chaotic_spde} throughout this article. 

\begin{assumption}[Assumptions on initial data and SPDE~\eqref{eq: chaotic_spde}] \label{ass: entropy} \
\begin{enumerate}[label = (A\arabic*)] 
\item CLT for initial values:
There exists \(\eta_0\), which belongs to the space of tempered distributions \(\mathcal{S}'(\mathbb{R}^d)\), such that the sequence \((\eta^N(0), N \in \N)\) converges in law to \(\eta(0)\) in \(H^{-\alpha}(\R^d)\) for \(\alpha > d/2+2\).  
\item There exists a pathwise unique solution \((\rho_t, t \ge 0)\) in the space \(H^a(\R^d)\) in the sense of Definition~\ref{def: solution_spde} above. 
 \item The relative entropy is uniformly bounded 
        \begin{equation}
           \sup\limits_{N \in \N} \sup\limits_{0 \le t \le T } \E \Big( \mathcal H(\rho_t^N \vert \rho_t^{\otimes N} )  \Big) < \infty . 
        \end{equation}
\end{enumerate}
\end{assumption}
\begin{remark} \label{remark: continuity_of_spde}
    From the existence of \(\rho\) we obtain that \(\rho \in C([0,T], H^{a-1}(\R^d)),\;\P\text{-a.e.}\) and 
    \begin{equation*}
    \E\big( \sup\limits_{0 \le t \le T} \norm{\rho_t}_{H^{a-1}(\R^d)}^2 \big) \le C 
    \end{equation*}
    by applying~\cite[Theorem~7.1]{Krylov1999AnAA}. 
\end{remark}

\begin{remark}\label{remark: compactness_ness}
 Notice that, in contrast to the article~\cite{Xianliang2023}, we require stronger convergence in \(H^{-\alpha}(\mathbb{R}^d)\) instead of \(\mathcal S'(\R^d)\). The main reason is that, by Prokhorov's theorem, the sequence \((\eta_0^N)_{N \in \mathbb{N}}\) is tight in \(H^{-\alpha}(\mathbb{R}^d)\) for every \(a > \alpha > d/2+2\). Even though we obtain uniform bounds on the \(H^{-\alpha}(\mathbb{R}^d)\)-norms, we cannot directly conclude compactness, since \(H^{-\alpha}(\mathbb{R}^d)\) is infinite-dimensional and the unit ball is not compact. 

 Instead, we need to identify a compact embedding. For a bounded smooth domain \(U\), the compact embedding \(H^{-\alpha}(U) \hookrightarrow H^{-\alpha'}(U)\) holds for \(\alpha' < \alpha\), but this fails in the case \(U = \mathbb{R}^d\). To overcome this difficulty in the unbounded domain setting, we impose this stronger assumption. 

An alternative approach would be to formulate our results in a different norm corresponding to weighted Sobolev spaces, where the weight is given by the Bessel potential \((I-\Delta)^{-a/2}\). Such spaces are frequently used in the mean-field community, and we refer to~\cite{Graham1996,Jourdain1998,Lacker2019} for a more detailed explanation.

\end{remark}
\begin{remark}
     Recently the author derived the relative entropy bound~\cite{nikolaev2024common} in our setting under some additional assumption on the initial condition \(\rho_0\) and coefficients \(\sigma, \nu\), utilizing the methods presented by Jabin Wang~\cite{JabinWangZhenfu2016,JabinWang2018}. 
Notice, that going through a bootstrap argument for the solution constructed in~\cite{nikolaev2024common}  with regular initial data \(\rho_0\) we can guarantee a solution satisfying the above Assumption. Since, regularity of non-local, non-linear SPDE's is outside the scope of this article, we refer to~\cite{Krylov1999AnAA,Hammersley2020,Veraar2022,nikolaev2023hk, nikolaev2024common} to some results on the well-posedness theory. 
\end{remark}

\begin{example}
An example for a kernel \(k\), which satisfies Assumption~\ref{ass: main} and~\ref{ass: entropy} is the Hegselmann--Krause kernel \(k = \indicator{|x| \le R }(x) |x| \) for some confidence radius \(R > 0\). The associated dynamics constitute an opinion formation model, more precisely a prototypical bounded confidence model, where agents only interact with others whose opinions lie within distance \(R\). 
See for instance~\cite{hegselmann2002,noorazar2020recent,nikolaev2023hk} for a more detailed analysis.
\end{example}

\begin{theorem}\label{theorem: main}[Main Theorem]
    Suppose that \((X_0^{i}, i \in \N)\) is i.i.d. and Assumptions~\ref{ass: main} and Assumptions~\ref{ass: entropy} hold. Then the sequence of measures \((\eta^N, N \in \N)\) of the interacting particle system~\eqref{eq: interacting_particle_system} converges in distribution in the space 
    \begin{equation*}
        L^2([0,T];H^{-\alpha}(\R^d)) \cap C([0,T]; H^{-\alpha-2}(\R^d))
    \end{equation*}
    for \( d/2 < \alpha <  a-2\) towards the unique solution \(\eta\) in the sense of Definition~\ref{def: pathwise_unique} of the fluctuation SPDE~\eqref{eq: limiting_spde}.  
\end{theorem}

%
%{\color{red}
%\begin{corollary}\label{cor: vortex}
%Let \(k_{\mathrm{vortex}} =  - \tfrac{1}{2\pi} \tfrac{x^{\perp}}{|x|^2}\) be the Biot--Savart law in \(\R^2\).
%Moreover let \((\rho_t, t \ge 0)\) be a solution of the Fokker--Planck equation~\eqref{eq: chaotic_spde} with \(\nu =0\) and \(\sigma > 0\) being a constant with regularity \(\rho \in C([0,T]; H^{a}(\R^d)) \cap C([0,T]; B_{\infty,\infty}^{a}(\R^d))\). 
%Moreover, let Assumption~\((A1)\) hold true. Then, \((\eta^N, N \in \N)\) converges in  
%\begin{equation*}
%        L^2([0,T];H^{-\alpha}(\R^d)) \cap C([0,T]; H^{-\alpha-2}(\R^d))
%    \end{equation*}
%for every \(\alpha > 1\) towards the unique solution in the sense of Definition~\ref{def: pathwise_unique} of the fluctuation SPDE~\eqref{eq: limiting_spde}. 
%\end{corollary}
%}

\subsection{Preliminary results}
Let us recall some crucial results from the seminal work~\cite{JabinWang2018,Xianliang2023} and provide an extension of the law of large numbers~\cite[Lemma~2.3]{Xianliang2023}. 
First, 
by the variational formula we can immediately obtain the following control 
\begin{equation} \label{eq: var_formula}
    \int_{\R^d} f \Id \mu \le \frac{1}{\kappa N } \bigg( \mathcal{H}(\mu \vert \tilde \mu) + \log\bigg( \int_{\R^d} \exp\big( \kappa N f\big) \Id \tilde \mu \bigg)\bigg) 
\end{equation}
for two probability measures \(\mu,\tilde \mu \), \(\kappa > 0\) and a bounded measurable function \(f\). For a proof we refer to~\cite[Lemma 2]{JabinWang2018}

A key component in the analysis of Gaussian fluctuations is the exponential law of large numbers established by Jabin and Wang~\cite[Theorem 4]{JabinWang2018} and further improved in~\cite[Remark 2.2]{Xianliang2023}, which also holds in the whole space \(\R^d\).

\begin{lemma} \label{lemma: jabinwang_exponential}
    For any probability measure \(\bar{\rho}\) on \(\mathbb{R}^d\), and any \(\varphi(x, y) \in L^\infty(\mathbb{R}^{2d})\) such that 
    \[
    \gamma := \tilde C \|\varphi\|^2_{L^\infty} < 1,
    \]
    with some fix constant \(\tilde C\). Assume that \(\varphi\) satisfies the following cancellations:
    \[
    \int_{\mathbb{R}^d} \varphi(x, y) \bar{\rho}(x) \, dx = 0 \quad \forall y, \quad 
    \int_{\mathbb{R}^d} \varphi(x, y) \bar{\rho}(y) \, dy = 0 \quad \forall x.
    \]
    Then  
    \[
    \sup_{N > 2} \int_{\mathbb{R}^{dN}}  \bar{\rho}^{\otimes N} \exp \big(N | \langle \varphi, \mu_N \otimes \mu_N \rangle| \big) \, d x 
    \leq \frac{2}{1 - \gamma} < \infty,
    \]
    where \(\mu_N = \frac{1}{N} \sum_{i=1}^N \delta_{x_i}\), \(x := (x_1, \dots, x_N) \in \mathbb{R}^{dN}\).
\end{lemma}

As in the seminal article~\cite{Xianliang2023}, we also require the modified version of the exponential law of large numbers~\cite[Lemma~2.3.]{Xianliang2023}. Again, the modification from the torus setting in~\cite{Xianliang2023} to the whole space \(\R^d\) presents no difficulties and, hence, we omit the proof at this step. 
\begin{lemma} \label{lemma: modification_elln}
Let \(\bar \rho\) be a probability measure on \(\R^d\). Assume further that functions \(\phi(x, y) \in L^\infty(\R^{d} \times \R^d)\) with \(\|\varphi\|_{L^\infty}\) small enough, and that 
\begin{equation*}
    \int_{\R^{d} \times \R^d} \bar \rho(x) \bar \rho(y) \phi(x, y) \, dx \, dy = 0.
\end{equation*}
Then 
\begin{equation*}
\int_{\R^{dN}} \bar \rho^{\otimes N} \exp \left( N \big| \langle \phi, \mu_N \otimes \mu_N \rangle \big|^2 \right) \, dx  
\leq 1 + \frac{\alpha_0}{1 - \alpha_0} + \frac{\beta_0}{1 - \beta_0},  
\end{equation*}
where 
\begin{equation*}
\alpha_0 := e^9 \|\phi\|_{L^\infty}^2 < 1, \quad \beta_0 := 4e \|\phi\|_{L^\infty}^2 < 1.
\end{equation*}
\end{lemma}

In order to use Kolmogorov's tightness criterion~\cite[Theorem~23.7]{Kallenberg2021} we require a small improvement of Lemma~\ref{lemma: modification_elln}, which will be proved in the Appendix. 

\begin{lemma} \label{lemma: modification_super}
Let \(\bar \rho\) be a probability measure on \(\R^d\). Assume further that the function \(\phi(x, y) \in L^\infty(\R^{d} \times \R^d)\) with \(\|\phi\|_{L^\infty}\) small enough, and that 
\begin{equation}
    \int_{\R^{d} \times \R^d} \bar \rho(y) \bar \rho(z) \phi(y, z) \Id y  \Id z= 0.
\end{equation}
Then 
\begin{equation*}
\int_{\R^{dN}} \bar \rho^{\otimes N} \exp \left( N \big| \langle \phi, \mu_N \otimes \mu_N \rangle \big|^{4} \right) \, \Id x 
\leq  C   . 
\end{equation*}

\end{lemma}

\section{Uniform estimates} \label{sec: uniform_estimates}
Following~\cite{Fernandez1997,Xianliang2023} we derive uniform estimates on the sequence of fluctuation process \((\eta^N, N \in\N)\) in appropriate negative Bessel potential spaces. By the nature of the Fourier transform on the torus as a infinite series which is given by the Fourier coefficients, we cannot directly apply the estimates derived in~\cite{Xianliang2023}. Hence, based on Parseval's identity we reproduce these estimates in the unbounded setting with common noise. First let us provide an adaptation of~\cite[Lemma~2.6]{Xianliang2023}.

\begin{lemma} \label{lemma: uniform_estimate}
Let \(m \in \{1,2\}\). 
For each \(\alpha > d/2\), we have the following inequality 
    \begin{equation}
         \sup\limits_{0 \le t \le T} \E\bigg( \norm{\mu_t^N - \rho_t }^{2m}_{H^{-\alpha}} \bigg)   \le  \frac{C(\alpha, m) }{N} \bigg(  \sup\limits_{0 \le t \le T}\E\bigg( \mathcal{H}(\rho_t^N \vert \rho_t^{\otimes N} \bigg)  + 1 \bigg) ,
    \end{equation}
    where the constant \(C(\alpha,m)\) depends on \(\alpha\) and \(m\).  
\end{lemma}
\begin{remark}
    We observe that the condition \(\alpha > d/2\) is necessary to ensure the integrability of the Bessel potential in the \(L^1(\R^d)\)-norm. 
\end{remark}
\begin{proof}
Utilizing the variational formula~\eqref{eq: var_formula} we find
\begin{align*}
     &\E\bigg( \norm{\mu_t^N - \rho_t }^{2m}_{H^{-\alpha}} \bigg)  \\
     &\quad =\E \bigg( \E \bigg(  \norm{\mu_t^N - \rho_t }^{2m}_{H^{-\alpha}} \vert \cF_t^W \bigg)  \bigg)  \\
     &\quad \le \frac{1}{\kappa N} \bigg( \E\big( \mathcal{H}(\rho_t^N \vert \rho_t^{\otimes N} \big) +  \E\bigg( \sup\limits_{0 \le t \le T} \log\bigg( \int_{\R^{dN} } \rho^{\otimes N}_s(x) \exp\bigg( \kappa N  \norm{\mu_t^N - \rho_t }^{2m}_{H^{-\alpha}}  \bigg) \Id x \bigg) \bigg) .
\end{align*}
    Now, by the crucial fact that the empirical measure is in \(H^{-\alpha}(\R^d)\) for \(\alpha > d/2\), we use Parseval's identity to find 
\begin{align*}
      \norm{\mu_t^N -\rho_t }_{H^{-\alpha}}^{2m} &=
      \norm{\mathcal{F}^{-1}( (1+|\xi|^2)^{-\alpha/2} \mathcal{F}(\mu_t^N -\rho_t)(\xi))}_{L^2(\R^d)}^{2m} \\
     &=  \norm{ (1+|\cdot |^2)^{-\alpha/2} \mathcal{F}(\mu_t^N -\rho_t)}_{L^2(\R^d)}^{2m}.
\end{align*}
This implies, 
\begin{align*}
    &\int_{\R^{dN}} \rho^{\otimes N}_s(x) \exp\bigg( \kappa N \norm{\mathcal{F}^{-1}( (1+|\xi|^2)^{-\alpha/2} \mathcal{F}(\mu_t^N -\rho_t)(\xi))}_{L^2}^{2m} \bigg) \Id x  \\
    &\quad =  \int_{\R^{dN}} \rho^{\otimes N}_s(x) \exp\bigg(  \kappa N 
 \norm{ (1+|\cdot|^2)^{-\alpha/2} \mathcal{F}(\mu_t^N -\rho_t)}_{L^2}^{2m} \bigg) \Id x \\
    &\quad =  \int_{\R^{dN}} \rho^{\otimes N}_s(x) \exp\bigg(  \kappa N 
 \bigg( \int_{\R^d} | \mathcal{F}(\mu_t^N -\rho_t)(\xi)|^2 (1+|\xi|^2)^{-\alpha}  \Id  \xi   \bigg)^m \bigg) \Id x \\
&\quad \le  \int_{\R^{dN}} \int_{\R^d} \exp\bigg(  \kappa N  \frac{ |\mathcal{F}(\mu_t^N -\rho_t)(\xi)|^{2m} }{ \norm{(1+|\cdot|^2)^{-\alpha}}_{L^1(\R^d)}^{-m}} \bigg)  \frac{(1+|\xi|)^{-\alpha}}{ \norm{(1+|\cdot|^2)^{-\alpha}}_{L^1(\R^d)}} \Id \xi  \;  \rho^{\otimes N}_s(x) \Id x \\ 
&\quad  \le  \sup\limits_{\xi \in  \R^d}
\int_{\R^{dN}}  \exp\Big(  \kappa N  \norm{(1+|\cdot|^2)^{-\alpha}}_{L^1(\R^d)}^{ m} |\mathcal{F}(\mu_t^N -\rho_t)(\xi)|^{2m}  \Big) \rho^{\otimes N}_s(x) \Id x ,  
\end{align*}
where we used Jensen's inequality. 
Next, we define the functions used in the exponential law of large numbers (Lemma~\ref{lemma: modification_elln}). Let \(\phi\) be given by 
\begin{align*}
    &\phi(t,\xi,z,\tilde{z}) \\
    &\quad := \bigg(\exp(-i z \cdot \xi)- \int_{\R^d} \exp(-i y \cdot \xi) \rho_t(y) \Id y \bigg)  \bigg(\exp(-i \tilde{z} \cdot \xi)- \int_{\R^d} \exp(-i y \cdot \xi) \rho_t(y) \Id y \bigg) .  
\end{align*}
Then, we have the following identity
\begin{equation}
     |\mathcal{F}(\mu_t^N -\rho_t)(\xi)|^{2m} 
     = |\langle \phi(t,\xi,\cdot ,\cdot ), \mu_t^N \otimes \mu_t^N \rangle|^m
\end{equation}
and since \(\rho_t\) is a provability measure, we have the trivial bound 
\begin{equation*}
    |\phi(t,\xi,z,\tilde{z})| \le  4 . 
\end{equation*}
It remains to check the cancellation property of Lemma~\ref{lemma: modification_elln}
\begin{align*}
    &\int_{\R^{2d}} \phi(t,\xi,z,\tilde{z}) \rho_t(z) \rho_t(\tilde z) \Id z  \Id \tilde z \\
    &\quad = \bigg( \int_{\R^d} \bigg(\exp(-i z \cdot \xi)- \int_{\R^d} \exp(-i y \cdot \xi) \rho_t(y) \Id y \bigg) \rho_t(z) \Id z \bigg)^2 =   0 ,  \quad \P\text{-a.e.}
\end{align*}
 Choosing \(\kappa = \big(8 \sqrt{e^9} \norm{(1+|\cdot|^2)^{-\alpha}}_{L^1(\R^d)}^m \big)^{-1} \) we obtain \[\max(e^9,4e) \big|\kappa  \norm{(1+|\cdot|^2)^{-\alpha}}_{L^1(\R^d)}^m \; \phi(t,\xi, z, \tilde z ) \big|^2 < 1.  \] Now, we can finally apply Lemma~\ref{lemma: jabinwang_exponential} in the case \(m=1\) and Lemma~\ref{lemma: modification_elln} in the case \(m=2\) to the function \(\kappa  \norm{(1+|\cdot|^2)^{-\alpha}}_{L^1(\R^d)}^m \phi \) to obtain 
\begin{align*}
     &\int_{\R^{dN}} \rho^{\otimes N}_s(x) \exp\bigg( \kappa N \norm{\mathcal{F}^{-1}( (1+|\xi|^2)^{-\alpha/2} \mathcal{F}(\mu_t^N -\rho_t)(\xi))}_{L^2}^4 \bigg) \Id x  \\
     &\quad \le \sup\limits_{N \ge 2 } \sup\limits_{\xi \in  \R^d}
\int_{\R^{dN}}  \exp\Big(  \kappa N  \norm{(1+|\cdot|^2)^{-\alpha}}_{L^1(\R^d)}^m |\langle \phi(t,\xi,\cdot ,\cdot ), \mu_t^N \otimes \mu_t^N \rangle|^2  \Big) \rho^{\otimes N}_s(x) \Id x  \\
&\quad \le  C 
\end{align*}
for some finite constant \(C\). 
Therefore, we find 
\begin{align*}
     \E\bigg( \norm{\mu_t^N - \rho_t }^{2m}_{H^{-\alpha}(\R^d)} \bigg) 
     \le \frac{1}{\kappa N} \bigg( \E(\mathcal{H}(\rho_t^N \vert \rho_t^{\otimes N} ) + C \bigg),
\end{align*}
which proves the claim. 
\end{proof}

We consider the interaction part in the SPDE formulation~\eqref{eq: etaN_spde}. Notice, that we explicitly need to identify the distribution of the Fourier transform. In order to estimate the term we follow~\cite[Lemma~2.8]{Xianliang2023}.

\begin{lemma} \label{lemma: uniform_interaction}
    Let \(\alpha > d/2+2\), then the inequality 
    \begin{align*}
        & \sup\limits_{0 \le t \le T}\E \bigg( \norm{\nabla \cdot ( k*\mu_t^N \mu_t^N -k *\rho_t \rho_t )}^4_{H^{-\alpha}(\R^d)} \bigg) \\
        &\quad \le \frac{C \big(\norm{k}_{L^\infty(\R^d)},\alpha \big)}{ N} \bigg(   \sup\limits_{0 \le t \le T}\E\bigg( \mathcal{H}(\rho_t^N \vert \rho_t^{\otimes N} \bigg) +1 \bigg)
    \end{align*}
    holds true. 
\end{lemma}
First, let us make sure the multiplication of \(k*\mu_t^N\) with \(\mu_t^N\) is well-defined as a Schwartz distribution. For \(\varphi \in \mathcal{S}(\R^d)\) we define 
\begin{equation*}
    \langle \nabla \cdot k*\mu_t^N \mu_t^N, \varphi \rangle =  - \langle \mu_t^N, k*\mu_t^N \cdot \nabla  \varphi \rangle,
\end{equation*}
where the right hand side bracket is to be understood as integration of \(k*\mu_t^N  \cdot \nabla \varphi \) with respect to \(\mu_t^N\). Here, we emphasize the important fact that \(\mu_t^N\) is a measure and, therefore, the right hand side is well-defined. Additionally, we can not replace the integration measure \(\mu_t^N\) by a general Schwartz distribution, since the product \(k*\mu_t^N \cdot \nabla \varphi \) is not smooth and therefore not a Schwartz distribution.

\begin{proof}
    Similar to Lemma~\ref{lemma: uniform_estimate} we start with inequality 
    \begin{align*}
     &\E\bigg( \norm{\nabla \cdot ( k*\mu_t^N \mu_t^N -k *\rho_t \rho_t )}^4_{H^{-\alpha}(\R^d)} \bigg)  \\
     &\quad \le \frac{1}{\kappa N} \bigg( \E\bigg( \sup\limits_{0 \le t \le T} \mathcal{H}(\rho_t^N \vert \rho_t^{\otimes N} \bigg) \\
     &\quad \quad +  \E\bigg(  \log\bigg( \int_{\R^{dN} } \rho^{\otimes N}_s(x) \exp\bigg( \kappa N   \norm{\nabla \cdot ( k*\mu_t^N \mu_t^N -k *\rho_t \rho_t )}^4_{H^{-\alpha}(\R^d)} \bigg) \Id x \bigg) \bigg) 
\end{align*}
for \(\kappa > 0\). 
    Utilizing Parseval's identity we have 
    \begin{equation*}
        \norm{\nabla \cdot( k*\mu_t^N \mu_t^N -k *\rho_t \rho_t )}^4_{H^{-\alpha}(\R^d)}  
        = \norm{ (1+|\xi|^2)^{-\alpha/2} \mathcal F(\nabla \cdot( k*\mu_t^N \mu_t^N -k *\rho_t \rho_t))}_{L^2(\R^d)}^4.
    \end{equation*}
    Let \(\varphi \in \mathcal S(\R^d)\), then we identify the Fourier transformation as follows 
    \begin{align*}
        \langle  \mathcal F(\nabla \cdot( k*\mu_t^N \mu_t^N ), \varphi \rangle
        &= -\langle  \mu_t^N , k*\mu_t^N \cdot \nabla \mathcal F ( \varphi) \rangle \\
        &= -\int_{\R^d}  \int_{\R^d} \nabla_{z} \exp(-i z \cdot \xi) \cdot k*\mu_t^N(z)   \varphi (\xi) \Id \xi \Id \mu_t^N(z) \\
        &= \langle i \int_{\R^d}  \exp(-i z \cdot \xi) \xi \cdot k*\mu_t^N(z)     \Id \mu_t^N(z ) ,\varphi \rangle .
    \end{align*}
    The same formula holds for \(\mu_t^N \) replaced by \(\rho_t\). We find  
    \begin{align*}
        \norm{\nabla \cdot( k*\mu_t^N \mu_t^N -k *\rho_t \rho_t )}^4_{H^{-\alpha}(\R^d)} 
        &= \bigg\|(1+|\xi|^2)^{-\frac{\alpha}{2}}\bigg( \int_{\R^d}  \exp(-i \xi \cdot z) \xi \cdot k*\mu_t^N(z)     \Id \mu_t^N(z) \\
        &\quad - \int_{\R^d}  \exp(-i \xi \cdot z) \xi  \cdot k*\rho_t(y)   \rho_t(y)  \Id y \bigg) \bigg\|_{L^2(\R^d)}^4 \\
        & \le\bigg \|(1+|\xi|^2)^{-\frac{\alpha-2}{2}}\bigg( \int_{\R^d}  \exp(-i \xi \cdot z)  k*\mu_t^N(z)     \Id \mu_t^N(z)\\
        & \quad  - \int_{\R^d}  \exp(-i \xi \cdot z)   k*\rho_t(y)   \rho_t(y)  \Id y \bigg) \bigg\|_{L^2(\R^d)}^4 
    \end{align*}
    As in Lemma~\ref{lemma: uniform_estimate} we define the function
    \begin{equation}
        \phi(t,\xi,z,\tilde z) : = \exp(-i \xi \cdot z)  k(z-\tilde z) -
        \int_{\R^d}  \exp(-i \xi \cdot y)  k*\rho_t(z)   \rho_t(y)  \Id y
    \end{equation}
    and notice that 
    \begin{align*}
        &\quad \bigg| \int_{\R^d}  \exp(-i \xi \cdot z)  k*\mu_t^N(z)     \Id \mu_t^N(z) - \int_{\R^d}  \exp(-i \xi \cdot z)  k*\rho_t(y)   \rho_t(y)  \Id y \bigg|^4 \\
        &= |\langle \phi(t,\xi,\cdot, \cdot), \mu_t^N \otimes \mu_t^N  \rangle|^4. 
    \end{align*}
    Additionally, we have the cancellation property 
    \begin{equation}
        \int_{\R^d\times \R^d } \phi(t,\xi,z, \tilde z ) \rho_t(z) \rho_t(\tilde z) \Id z \Id \tilde z = 0 
    \end{equation}
    and the uniform bound 
    \begin{equation*}
        |\phi(t,\xi, z,\tilde z)| \le 2\norm{k}_{L^\infty(\R^d)} . 
    \end{equation*}
    Notice that the cancellation property also holds for the real and imaginary part of the integrand by the simple fact that we can interchange real and imaginary part with the integral operator. Putting everything together, applying Jensen inequality, Fubini's theorem and Lemma~\ref{lemma: modification_elln} we find 
    \begin{align*}
        &\int_{\R^{dN} } \rho^{\otimes N}_s(x) \exp\bigg( \kappa N   \norm{\nabla \cdot ( k*\mu_t^N \mu_t^N -k *\rho_t \rho_t )}^4_{H^{-\alpha}(\R^d)} \bigg) \Id x \\
        &\quad \le \int_{\R^{dN} } \rho^{\otimes N}_s(x) \exp\bigg( \kappa N \bigg(  \int_{\R^d} (1+|\xi|^2)^{-\alpha+2} |\langle \phi(t,\xi,\cdot, \cdot), \mu_t^N \otimes \mu_t^N  \rangle|^2  \Id \xi \bigg)^2 \bigg)  \Id x \\
        &\quad \le   \sup\limits_{\xi \in \R^d} \int_{\R^{dN} } \rho^{\otimes N}_s(x) \exp\bigg( \kappa N \norm{ (1+|\cdot|^2)^{-\alpha+2} }_{L^1(\R^d)}^2 |\langle \phi(t,\xi,\cdot, \cdot), \mu_t^N \otimes \mu_t^N  \rangle|^4\bigg)  \Id x  \\
        &\quad \le C\big(\norm{k}_{L^\infty(\R^d)},\alpha \big) ,
    \end{align*}
    where we choose \(\kappa\) small enough such that \(\kappa \norm{ (1+|\cdot|^2)^{-\alpha+2} }_{L^1(\R^d)}  \phi \) satisfy the smallness condition of Lemma~\ref{lemma: modification_super}. Finally, we arrive at 
    \begin{align*}
        &\E\bigg( \sup\limits_{0 \le t \le T}  \norm{\nabla \cdot ( k*\mu_t^N \mu_t^N -k *\rho_t \rho_t )}^2_{H^{-\alpha}(\R^d)} \bigg)
        \\
        &\quad \le \frac{C\big(\norm{k}_{L^\infty(\R^d)},\alpha \big)}{ N} \bigg( \E\bigg( \sup\limits_{0 \le t \le T} \mathcal{H}(\rho_t^N \vert \rho_t^{\otimes N} \bigg) +1 \bigg). 
    \end{align*}
\end{proof}

 The following bound will be very useful in the future we reproduce~\cite[Lemma~2.9]{Xianliang2023}. 
\begin{lemma} \label{lemma: uniform_identification}
Let \(\varphi \colon \R^d \mapsto \R^d \) be bounded and continuous. Then there exists a constant \(C(\varphi)\) depending on \(\varphi\) such that 
    \begin{equation*}
        \E ( |\langle \varphi \cdot  k*(\mu_t^N -\rho_t),\mu_t^N -\rho_t \rangle | ) 
        \le \frac{C(\varphi)}{N} \bigg( \E\bigg( \sup\limits_{0 \le t \le T} \mathcal{H}(\rho_t^N \vert \rho_t^{\otimes N} \bigg) +1 \bigg). 
    \end{equation*}
\end{lemma}
\begin{proof}
    Defining 
    \begin{equation*}
        \phi(t,x,y) := k(x-y)\cdot  \varphi(x) - \varphi(x) \cdot k*\rho_t(x) - \langle k(\cdot-y),\rho_t \rangle + \langle \varphi \cdot k *\rho_t,\rho_t \rangle 
    \end{equation*}
    and observe that 
    \begin{equation*}
        \langle \varphi\cdot  k*(\mu_t^N -\rho_t),\mu_t^N -\rho_t \rangle 
        = \langle \phi(t,\cdot, \cdot), \mu_t^N \otimes \mu_t^N \rangle 
    \end{equation*}
    and \(| \phi(t,x,y)| \le 4 \norm{k}_{L^\infty(\R^d)} \norm{\varphi}_{L^\infty(\R^d)} \). 
    We also have the cancellation properties 
    \begin{equation*}
        \int_{\R^d} \phi(t,x,y) \rho_t(x) \Id x =0 , \quad \int_{\R^d} \phi(t,x,y) \rho_t(y) \Id y=0.
    \end{equation*}
    Applying inequality~\eqref{eq: var_formula} and Lemma~\ref{lemma: jabinwang_exponential} in a similar fashion as Lemma~\ref{lemma: uniform_estimate} we obtain 
    \begin{equation*}
         \E ( |\langle \varphi \cdot  k*(\mu_t^N -\rho_t),\mu_t^N -\rho_t \rangle | ) 
        \le \frac{C}{N} \bigg( \E\bigg( \sup\limits_{0 \le t \le T} \mathcal{H}(\rho_t^N \vert \rho_t^{\otimes N} \bigg) +1 \bigg). 
    \end{equation*}
\end{proof}

\begin{remark} \label{remark: vortex_uniform}
As mentioned the above Lemmas are adaptations of~\cite[Lemma~2.6, Lemma~2.8, Lemma~2.9]{Xianliang2023}. Hence, utilizing similar techniques as in the Lemmas mentioned above we can extend every Lemma in this section to the case where \(\nu = 0\) and the interaction kernel is given by Biot--Savart kernel \(k_{\mathrm{vortex}}\). The main technique is a symmetrization technique in combination with the relative entropy estimates provided by~\cite{Feng2023} on the whole space \(\R^d\). Notice, that in this case we do not utilize \(k \in L^\infty(\R^d)\).  
\end{remark}

\section{Tightness of fluctuation measure} \label{sec: tightness}
In this section we utilize the findings from Section~\ref{sec: uniform_estimates} to demonstrate the tightness of \((\eta^N, N \in \N)\). 
For each \(\varphi \in S(\R^d)\) We define the martingale part sequence as 
\begin{equation}\label{eq: martingal_sequence}
    \mathcal M_t^N(\varphi) := \frac{1}{\sqrt{N}} \sum\limits_{i=1}^N \int\limits_0^t  ( \sigma^{\mathrm{T}}(s,X_s^{i}) \nabla \varphi(X_s^{i}) ) \Id B_s^{i}
\end{equation}
and the common noise sequence as 
\begin{equation} \label{eq: common_noise_martingale_sequence}
\hat{\mathcal M}_t^N (\varphi) :=  \sum\limits_{j=1}^d \sum\limits_{\tilde{l}=1}^{\tilde{m}} \int\limits_0^t   \langle \nu^{j,\tilde{l}}(s,\cdot)  \partial_{z_j} \varphi(\cdot), \eta_s^N \rangle  \Id W_s^{\tilde{l}} . 
\end{equation}
Similar to~\cite{Xianliang2023} we need to define both integrals as measurable maps into the space \(H^{-\alpha}\), i.e. the maps 
\begin{equation*}
     \mathcal M_t^N \colon \Omega \mapsto H^{-\alpha} \quad \mathrm{and} \quad  \hat{\mathcal M}_t^N\colon \Omega \mapsto H^{-\alpha}
\end{equation*}
need to be strongly measurable. Notice, that \(\hat{\mathcal M}_t^N\) has a stochastic integral representation and, therefore, it will be enough to verify the square integrability condition in  \(H^{-\alpha}\). Indeed, this is true for \(\alpha> d/2+1\).
\begin{lemma}\label{lemma: versions_of_stochastic_integrals}
    For \(\alpha > d/2+1\) and for each \(N \in \N\) there exists a progressively measurable processes \(\mathcal{M}_t^N \), \(\hat{\mathcal M}_t^N \) with values in \(H^{-\alpha}(\R^d)\) such that~\eqref{eq: martingal_sequence} and~\eqref{eq: common_noise_martingale_sequence} hold almost surely for all \(t \ge 0\) and \(\varphi \in H^\alpha(\R^d)\). 
\end{lemma}
We postpone the proof to the Appendix. 

\begin{lemma} \label{lemma: tightness_martingale}
    For every \(\alpha > d/2 +1\) the sequence \((\mathcal M^N, N \in \N)\) is tight in the space \(C([0,T];H^{-\alpha})\) and the following inequalities
    \begin{align*}
        \E\big(\norm{\mathcal M_t^N - \mathcal{M}_s^N}_{H^{-\alpha} /\R^d) }^{2\theta} \big) 
        &\le  C(d,m,\alpha_2,\theta', \sigma) |t-s|^\theta , \quad 0 \le s \le t \le T \\
        \E\big( \sup\limits_{0 \le t \le T} \norm{\mathcal M_t^N}_{H^{-\alpha}(\R^d) }^{2\theta} \big) 
        &\le  C(d,m,\alpha_2,\theta', \sigma,T). 
    \end{align*}
    holds true for \(\theta > 1\).   
\end{lemma}
\begin{proof}
    Let us look at the sequence \((\mathcal M_t^N, N \in \N)\) under the Fourier transformation as a Schwarz distribution on the space \(H^{\alpha}\). Hence, let \(\varphi \in \testfunctions{\R^d}\), then applying \(\mathcal M_t^N\) on \(\varphi\) we obtain 
    \begin{align}
    \begin{split}
       \langle  \mathcal{F}(\mathcal{M}_t^N),\varphi\rangle &= 
       \langle \mathcal M_t^N , \mathcal{F}(\varphi) \rangle \\
       &= \frac{1}{\sqrt{N}} \sum\limits_{i=1}^N  \int\limits_0^t  ( \sigma^{\mathrm{T}} \nabla \mathcal{F}\varphi )(X_s^{i}) \Id B_s^{i}\\
       &=  \frac{1}{\sqrt{N}} \sum\limits_{i=1}^N \sum\limits_{j=1}^d  \sum\limits_{l=1}^m  \int\limits_0^t \sigma^{j,l}(t,X_t^{i})  \partial_{\xi_j}  \mathcal{F}(\varphi )(X_s^{i}) \Id B_s^{i,l} \\
       &=  \frac{1}{\sqrt{N}} \sum\limits_{i=1}^N \sum\limits_{j=1}^d  \sum\limits_{l=1}^m  \int\limits_0^t \sigma^{j,l}(t,X_t^{i})  \mathcal{F}( -iz_j  \varphi )  (X_s^{i}) \Id B_s^{i,l} \\
       &= -\frac{1}{\sqrt{N}} \sum\limits_{i=1}^N \sum\limits_{j=1}^d  \sum\limits_{l=1}^m \int\limits_0^t  \int_{\R^d} \sigma^{j,l}(t,X_t^{i}) \exp(-i X_s^{i} \cdot z) iz_j  \varphi(z) \Id z    \Id B_s^{i,l} . 
    \end{split}
    \end{align}
    
    Now, we want to apply the stochastic Fubini theorem~\cite{veraar2012}. We check the integrability condition
    \begin{equation*}
        \int_{\R^d} \bigg( \int\limits_0^T | \sigma^{j,l}(t,X_t^{i}) \exp(-i X_s^{i} \cdot z) iz_j  \varphi(z) |^2 \Id s \bigg)^{\frac{1}{2}} \Id z  
        \le C(T,\sigma)  \int_{\R^d} |z| |\varphi(z)| \Id z < \infty.
    \end{equation*}

    Hence, we can interchange the stochastic integral with the Lebesgue integral and arrive at 
    \begin{equation*}
         \langle  \mathcal{F}(\mathcal{M}_t^N),\varphi\rangle 
         = \bigg\langle -\frac{1}{\sqrt{N}} \sum\limits_{i=1}^N \sum\limits_{j=1}^d  \sum\limits_{l=1}^m \int\limits_0^t  \sigma^{j,l}(t,X_t^{i}) \exp(-i X_s^{i} \cdot z ) iz_j     \Id B_s^{i,l} ,  \varphi(z) \bigg \rangle_{L^2(\R^d)}.
    \end{equation*}
    This means the Fourier transformation of \(\mathcal M_t^N \) is given by the left expression in the right bracket for \(\varphi \in \testfunctions{\R^d}\). Since \(\testfunctions{\R^d}\) is dense in \(H^{\alpha}(\R^d)\) we obtain a unique extension of the operator. Notice, that this operator is also explicitly given for smooth functions \(\varphi\) such that \(|\cdot| \varphi \in L^1(\R^d)\). 
    Next, let \(\theta > 1\), \(1/\theta + 1/ \theta' =1\), \(\alpha_1+\alpha_2 = 2\alpha \)
    Let us use Parseval's identity again and the fact that \(2\alpha > d +2\) to obtain 
    \begin{align} \label{eq: compactness_martingale_aux}
        &\norm{\mathcal M_t^N - \mathcal{M}_s^N}_{H^{-\alpha} }^{2\theta} \nonumber \\
        &\quad= \bigg(\int_{\R^d} (1+|\xi |^2)^{-(\alpha_1+\alpha_2)} |\mathcal{F}(\mathcal M_t^N - \mathcal{M}_s^N)|^2(\xi) \Id \xi \bigg)^{\theta} \nonumber \\
        &\quad\le \bigg(\int_{\R^d} (1+|\xi |^2)^{-\alpha_1 \theta } |\mathcal{F}(\mathcal M_t^N - \mathcal{M}_s^N)|^{2\theta}(\xi) \Id \xi \bigg)
        \bigg( \int_{\R^d} (1+|\xi |^2)^{-\alpha_2 \theta' }\Id \xi \bigg)^{\frac{1}{\theta'}} \nonumber \\
        &\quad\le C(d,m) \sum\limits_{j=1}^d \sum\limits_{l=1}^m  \bigg(\int_{\R^d} (1+|\xi |^2)^{-\alpha_1 \theta } \bigg | \frac{1}{\sqrt{N}} \sum\limits_{i=1}^N  \int\limits_s^t  \sigma^{j,l}(t,X_t^{i}) \exp(-i X_s^{i} \cdot \xi ) i \xi_j     \Id B_s^{i,l} \bigg|^{2\theta} \Id \xi \bigg) \nonumber  \\
        & \quad \quad \cdot 
        \bigg( \int_{\R^d} (1+|\xi |^2)^{-\alpha_2 \theta' }\Id \xi \bigg)^{\frac{1}{\theta'}} 
    \end{align}
    In the following we need to choose \(\alpha_2' \) with \(\alpha_2 \theta' > d\) such that the last term is integrable. 
    We keep this condition in mind and will choose explicit parameters at the end. 
    For the stochastic integral we apply the BDG inequality and the exchangability of the particle system to obtain 
    \begin{align*}
        &\E\bigg(  \bigg | \frac{1}{\sqrt{N}} \sum\limits_{i=1}^N  \int\limits_s^t  \sigma^{j,l}(t,X_t^{i}) \exp(-i X_s^{i} \cdot \xi) i \xi_j    \Id B_s^{i,l} \bigg|^{2\theta} \bigg)  \\
        & \quad \le  \E\bigg( \bigg( \frac{1}{N} \sum\limits_{i=1}^N  \int\limits_s^t  |\sigma^{j,l}(t,X_t^{i}) \exp(-i X_s^{i} \cdot \xi ) i \xi_j |^2    \Id s \bigg)^{\theta} \bigg)  \\
        &\quad \le C(\sigma) |t-s|^\theta |\xi|^{2\theta}    
    \end{align*}
    
    Consequently, taking the expected value in inequality~\eqref{eq: compactness_martingale_aux} and afterwards Fubini's theorem, we can apply the previous observation to obtain 
    \begin{align*}
        \E\big(\norm{\mathcal M_t^N - \mathcal{M}_s^N}_{H^{-\alpha} }^{2\theta} \big) 
        &\le C(d,m,\alpha_2,\theta', \sigma) |t-s|^\theta \int_{\R^d} |\xi|^{2\theta}    (1+|\xi |^2)^{-\alpha_1 \theta } \Id \xi \\
        &\le  C(d,m,\alpha_2,\theta', \sigma) |t-s|^\theta \int_{\R^d}   (1+|\xi |^2)^{(-\alpha_1 +2) \theta  } \Id \xi . 
    \end{align*}
    It remains to assure that the parameters satisfy integrability of the integrals. Hence \(\alpha_1, \alpha_2, \theta, \theta'\) need to fulfill the following conditions 
    \begin{itemize}
        \item \(\alpha_2 \theta' > d \), \((\alpha_1-2) \theta > d, \theta> 1 \), 
        \item \(\alpha_1 + \alpha_2 = 2\alpha  \), \(1/\theta + 1/\theta' = 1\).
    \end{itemize}
    Choosing \(\alpha_1 = \alpha +1 - \frac{d}{2}+ \frac{d}{\theta}\), \(\alpha_ 2 =  \alpha - 1 + \frac{d}{2}- \frac{d}{\theta}\) and arbitrary \(\theta>1\) satisfies all conditions and the integral is finite. 
    Replicating the computations with \(s=0\) and the supremum inside the expectation, we obtain the remaining inequality. 
\end{proof}

\begin{lemma}[Almost tightness of common noise] \label{lemma: tightness_common_noise}
For every \(\alpha >d/2+1\) the sequence \((\tilde M_t^N, N \in \N)\) satisfies the inequalities 
  \begin{align*}
        \E\big(  \norm{\hat {\mathcal  M}_t^N- \hat{\mathcal M}_s^N}_{H^{-\alpha}}^{2\theta}\big)
        &\le  C(d,\tilde m,\nu)  \E \bigg( \bigg( \int\limits_{s}^t \norm{\eta_r}_{H^{-\alpha+1}(\R^d)}^{2}  \Id r \bigg)^{\theta} \bigg), \\
        \E \big( \sup\limits_{0\le t \le T} \norm{\hat{\mathcal M}_t^N}_{H^{-\alpha}}^{2\theta}\big)
        &\le  C(d,\tilde m,\nu,\alpha ) \bigg(  \sup\limits_{0 \le t \le T}\E\bigg( \mathcal{H}(\rho_t^N \vert \rho_t^{\otimes N} \bigg)  + 1 \bigg) , 
    \end{align*}
    for \(s < t \) and arbitrary \(\theta>0\).  

\end{lemma}
\begin{proof}
    By Lemma~\ref{lemma: uniform_estimate} we have 
    \begin{align} \label{eq: common_noise_tigthness_aux1}
    \begin{split}
        \sup\limits_{0 \le t \le T } \E\bigg( \norm{ \eta^N_t}_{H^{-\alpha'}}^2 \bigg) 
        &= N \sup\limits_{0 \le t \le T }  \E\bigg(\norm{ \mu_t^N -\rho_t}_{H^{-\alpha'}}^2 \bigg)  \\
        &\le C(\alpha) \sup\limits_{0 \le t \le T} \bigg(  \E\bigg(\mathcal{H}(\rho_t^N \vert \rho_t^{\otimes N} \bigg)  + 1 \bigg) . 
        \end{split}
    \end{align}
    for any \(d/2 < \alpha' < \alpha \). 
    Together with the assumption on the uniform bound of the relative entropy, we have the uniform bound on \(\eta_t^N\) in \(H^{-\alpha'}(\R^d)\).
    Consequently, utilizing the Burkholder-Davis-Gundy inequality for Banach space valued martingales~\cite[Theorem~1.1]{Marinelli2016}, Theorem~\ref{theorem: holder} and Assumption~\ref{ass: main} we obtain 
    \begin{align*}
       \E\bigg( \norm{\hat {\mathcal  M}_t^N -\hat{\mathcal M}_s^N}_{H^{-\alpha}(\R^d)}^{2\theta} \bigg) 
       &\le \sum\limits_{j=1}^d \sum\limits_{\tilde l}^{\tilde m} \E \bigg( \norm{\int\limits_s^t   \partial_{z_j}( \nu^{j,\tilde l}(s,\cdot) \eta_r^N ) \Id W_r^{\tilde l}}_{H^{-\alpha}(\R^d)}^{2\theta} \bigg)\\
        &\le \sum\limits_{j=1}^d \sum\limits_{\tilde l}^{\tilde m} \E \bigg( \bigg( \int\limits_s^t  \norm{  \partial_{z_j}( \nu^{j,\tilde l}(s,\cdot) \eta_r^N ) }_{H^{-\alpha}(\R^d)}^2 \Id r  \bigg)^{\theta} \bigg) \\
       &\le C(\nu) \E \bigg( \bigg( \int\limits_{s}^t \norm{\eta_r^{ N}}_{H^{-\alpha+1}(\R^d)}^{2}  \Id r \bigg)^{\theta} \bigg) 
    \end{align*}
Carrying out the same computation with \(s=0\) and the supremum inside the expectation we obtain the second inequality of the statement by applying inequality~\eqref{eq: common_noise_tigthness_aux1} at the end.  

\end{proof}
Next, let us improve the bound in Lemma~\ref{lemma: uniform_estimate} to include the supremum inside the expected value.  

\begin{lemma} \label{lemma: uniform_estimate_sup}
Let \(\alpha > d/2 +2\), then we have the bound 
    \begin{equation*}
        \E\big( \sup\limits_{0 \le t \le T} \norm{\eta_t^N}_{H^{-\alpha}(\R^d)}^4 \big) 
        \le  C(d,\sigma,\nu, \alpha, T). 
    \end{equation*}
\end{lemma}
\begin{proof}
    We decompose the norm as follows 
    \begin{equation*}
        \norm{\eta_t^N -\eta_s^N}_{H^{-\alpha}(\R^d)}^4 \le  \sum\limits_{j=1}^4 J^{j}_{s,t},
    \end{equation*}
    where 
    \begin{align*}
        J^1_{s,t} : &= \left\|  \int_s^t \frac{1}{2}  \sum\limits_{\alpha,\beta=1}^d \partial_{z_{\alpha}}\partial_{z_{\beta}} \bigg( ( [\sigma_r\sigma_r^{\mathrm{T}}]_{(\alpha,\beta)} + [\nu_r \nu_r^{\mathrm{T}}]_{(\alpha,\beta)} ) \eta_r^N \bigg) \Id r \right\|_{H^{-\alpha}(\R^d)}^4, \\
    J^2_{s,t} : &= N  \left\|  \int_s^t  \nabla \cdot ( k*\mu_r^N \mu_r^N -k *\rho_r \rho_r ) \Id r \right\|_{H^{-\alpha}(\R^d)}^4, \\
    J^3_{s,t} :& =  \norm{ \mathcal M_t^N - \mathcal M_s^N }_{H^{-\alpha}(\R^d)}^4 , \quad 
    J^4_{s,t} :  = \norm{\hat{\mathcal M}_t^N- \hat{\mathcal M}_s^N}_{H^{-\alpha}(\R^d)}^4
    \end{align*}
    Here we used the characterization of the negative Bessel potential \(H^{-\alpha }(\R^d) \) space as the dual space of the space \(H^\alpha(\R^d)\) and that the testfunctions \(\testfunctions{\R^d}\) are dense in \(H^\alpha\) and therefore norming and we can use the expansion~\eqref{eq: etaN_spde}. 
    At the moment, we only require the case \(s=0\). 
    Utilizing~\cite[Proposition~1.2.2.]{Neerven2016} and the Pointwise Multiplier Theorem~\ref{theorem: holder}, we obtain 
    \begin{align*}
        \E \big( \sup\limits_{0 \le t \le T} J_{0,t}^1 \big) 
        &\le T \E \bigg( \int\limits_0^T \norm{\sum\limits_{\alpha,\beta=1}^d \partial_{z_{\alpha}}\partial_{z_{\beta}} \bigg( ( [\sigma_t\sigma_t^{\mathrm{T}}]_{(\alpha,\beta)} + [\nu_t \nu_t^{\mathrm{T}}]_{(\alpha,\beta)} ) \eta_t^N \bigg) }_{H^{-\alpha}(\R^d)}^4 \Id t  \bigg) \\
   &\le C(d) T \sum\limits_{\alpha,\beta=1}^d   \sup\limits_{0 \le t \le T}  \E \big( \norm{ ( [\sigma_t\sigma_t^{\mathrm{T}}]_{(\alpha,\beta)} + [\nu_t \nu_t^{\mathrm{T}}]_{(\alpha,\beta)} ) \eta_t^N }_{H^{-\alpha+2}(\R^d)}^4 \Id t \big) \\
   &\le C(d) T  \sup\limits_{0 \le t \le T} \norm{[\sigma_t\sigma_t^{\mathrm{T}}]_{(\alpha,\beta)} + [\nu_t \nu_t^{\mathrm{T}}]_{(\alpha,\beta)}}_{B_{\infty, \infty}^{a}(\R^d)}^{4}  \E\bigg( \norm{\eta_t^N }_{H^{-\alpha+2}(\R^d)}^{4}\bigg) \\
   &\le C(d, \sigma,\nu,\alpha) T  \sup\limits_{0 \le t \le T} \bigg(  \E\bigg(\mathcal{H}(\rho_t^N \vert \rho_t^{\otimes N} \bigg)  + 1 \bigg) 
    \end{align*}
     For the term \(J^2_{0,t}\) we apply H\"older's inequality to obtain 
 \begin{align*}
     \E(J_{0,t}^2) &\le T^4 N \sup\limits_{0 \le t \le T} \E\bigg(  \norm{\nabla \cdot ( k*\mu_t^N \mu_t^N -k *\rho_t \rho_t )}^4_{H^{-\alpha}(\R^d)}  \bigg) \\ 
     &\le T^2  C \big(\norm{k}_{L^\infty(\R^d)},\alpha \big)  \bigg( \E\bigg( \sup\limits_{0 \le t \le T} \mathcal{H}(\rho_t^N \vert \rho_t^{\otimes N} \bigg) +1 \bigg),
 \end{align*}
 which, through the bound on the relative entropy, establishes the desired bound
  \begin{equation*}
     \E \big( \sup\limits_{0\le t \le T} \norm{\eta_t^N}_{H^{-\alpha}(\R^d)}^{4} \big) \le \E \big((\norm{\eta_0^N}_{H^{-\alpha}(\R^d)}^4\big) + C(d,\sigma,\nu, \alpha, T) 
     \le C(d,\sigma,\nu, \alpha, T),
 \end{equation*}
 where we applied Lemma~\ref{lemma: uniform_estimate} to estimate the fluctuation process and initial time.
 
 The bounds for \(J_{0,t}^3\) and \(J_{0,t}^4\) follow immediately by Lemma~\ref{lemma: tightness_martingale} and Lemma~\ref{lemma: tightness_common_noise}. 
 Putting all estimates together, the claim follows. 
\end{proof}

\begin{lemma} \label{lemma: tighntess_fluctuation_sequence}
Let \(\alpha > d/2 +2 \). Then, the sequence of fluctuation processes \((\eta^N, N \in \N)\) is tight in the space \(C([0,T]; H^{-\alpha}(\R^d))\). 
\end{lemma}
\begin{proof}
     We rely on the decomposition given by Lemma~\ref{lemma: uniform_estimate_sup}
     \begin{equation*}
        \norm{\eta_t^N -\eta_s^N}_{H^{-\alpha}(\R^d)}^4 \le  \sum\limits_{j=1}^4 J^{j}_{s,t}.  
    \end{equation*}
Since the computations are similar to Lemma~\ref{lemma: uniform_estimate_sup} we skip some steps. 
We apply~\cite[Proposition~1.2.2.]{Neerven2016} to obtain 
\begin{align*}
   \E( J_{s,t}^1 ) &\le |t-s|^4 \E \bigg( \sup\limits_{0 \le t \le T} \norm{\sum\limits_{\alpha,\beta=1}^d \partial_{z_{\alpha}}\partial_{z_{\beta}} \bigg( ( [\sigma_t\sigma_t^{\mathrm{T}}]_{(\alpha,\beta)} + [\nu_t \nu_t^{\mathrm{T}}]_{(\alpha,\beta)} ) \eta_t^N \bigg) }_{H^{-\alpha}(\R^d)}^4  \bigg) \\
   &\le C(d)|t-s|^4 \sup\limits_{0 \le t \le T} (\norm{[\sigma_t\sigma_t^{\mathrm{T}}]_{(\alpha,\beta)} + [\nu_t \nu_t^{\mathrm{T}}]_{(\alpha,\beta)}}_{\mathcal B^{a}_{\infty, \infty}(\R^d)}\big)  E\big(\sup\limits_{0 \le t \le T}  \norm{\eta_t^N }_{H^{-\alpha+2}(\R^d)}^{4} \big) \\
   &\le C(d, \sigma,\nu,\alpha,T)|t-s|^2          
\end{align*}
where we applied Lemma~\ref{lemma: uniform_estimate_sup} and Assumption~\ref{ass: main} in the last step.  
 For the term \(J^2_{s,t}\) we use Lemma~\ref{lemma: uniform_interaction} we find  
 \begin{align*}
     \E(J_{s,t}^2) &\le |t-s|^3 \E\bigg( \int\limits_{s}^t  \norm{\nabla \cdot ( k*\mu_r^N \mu_r^N -k *\rho_r \rho_r )}^4_{H^{-\alpha}(\R^d)} \Id r  \bigg) \\ 
      &\le |t-s|^4 \sup\limits_{0 \le t \le T} \E\bigg(  \norm{\nabla \cdot ( k*\mu_t^N \mu_t^N -k *\rho_t \rho_t )}^{ 4}_{H^{-\alpha}(\R^d)}   \bigg) \\ 
     &\le C\big(\norm{k}_{L^\infty(\R^d)},\alpha,T \big) |t-s|^2  \sup\limits_{0 \le t \le T} \bigg( \E\bigg( \mathcal{H}(\rho_t^N \vert \rho_t^{\otimes N} \bigg) +1 \bigg),
 \end{align*}
 which, through the bound on the relative entropy, establishes the desired condition.
 For \(J_{s,t}^3\) we refer to the inequality Lemma~\ref{lemma: tightness_martingale}. 
 For the common noise term \(J_{s,t}^4\) we use Lemma~\ref{lemma: uniform_estimate_sup} with Lemma~\ref{lemma: tightness_common_noise}. 
Plugging the estimate of Lemma~\ref{lemma: uniform_estimate_sup} in the first inequality in Lemma~\ref{lemma: tightness_common_noise} we derive 
 \begin{equation*}
     \E(J^4_{s,t}) \le  C(d,\tilde m,\nu)  \E \bigg(\bigg(\int\limits_{s}^t \norm{\eta_r}_{H^{-\alpha+1}(\R^d)}^{2}  \Id r \bigg)^{2} \bigg)
     \le  C(d,\tilde m,\nu)  \E \big( \sup\limits_{0\le t \le T} \norm{\eta_t^N}_{H^{-\alpha}(\R^d)}^4 \big) |t-s|^2. 
 \end{equation*} 
 Consequently, we find 
\begin{align*}
     \E \big( \norm{\eta_t^N -\eta_s^N}_{H^{-\alpha}(\R^d)}^{4} \big) 
     &\le \sum\limits_{j=1}^4 \E ( J^j_{s,t} ) \\
     &\le C(d,\sigma,\nu,\tilde m,m, \alpha, T) |t-s|^2  . 
\end{align*}

 Together with Assumption~\ref{ass: main} and Kolmogorov's tightness criterion~\cite[Theorem~23.7.]{Kallenberg2021} we obtain the tightness of \((\eta_t^N, N \in \N) \) in \(C([0,T],  H^{-\alpha}(\R^d)  )\). 
\end{proof}

Let us introduce the spaces
\begin{align}\label{eq: spaces}
\begin{split}
\mathcal W &:= C\big( [0,T]; \R \big),\\
\mathcal X &:=  \bigcap\limits_{m \in \N} C\big([0,T], H^{-\frac{d}{2}-2-\frac{1}{m} } (\R^d) \big)  \cap L^2\big([0,T], H^{-\frac{d}{2}-\frac{1}{m} }(\R^d) \big) ,  \\
\mathcal Y &:= \bigcap\limits_{m \in \N} C\big([0,T], H^{-\frac{d}{2}-2-\frac{1}{m} } (\R^d) \big)  , \\
\mathcal Z &:= C([0,T],H^{a-1}(\R^d)). 
\end{split}
\end{align}
Here \(\mathcal W\) correspond to the space of the common noise \(W\), \(\mathcal X\) is the space for the fluctuation process \(\eta^N\), \(\mathcal Y\) is the space for the martingale part and \(\mathcal Z\) is the space for the solution of the SPDE \(\rho\),

We demonstrate that the quadruple \(((W,\eta^N,\mathcal M^N, \rho), N \in \N)\) is tight following~\cite[Theorem~3.5]{Xianliang2023}.

\begin{lemma}
The law of \(((W,\eta^N,\mathcal M^N, \rho), N \in \N)\) is tight in \(\mathcal W \times \mathcal X \times \mathcal Y \times \mathcal Z \). 

\end{lemma}

\begin{proof}
    It is enough to demonstrate that each process is tight in their corresponding Polish space. Then, the vector of process is tight by the fact that the cartesian product of compact sets is compact under the product topology. The fact that \((\mathcal M^N, N \in \N)\) is tight follows by Lemma~\ref{lemma: tightness_martingale}. 
    Indeed, similar to~\cite[Theorem 3.5]{Xianliang2023} it is enough to demonstrate that \((\mathcal M^N, N \in \N)\) is tight in \( C\big([0,T], H^{-\frac{d}{2}-2-\frac{1}{m} } \big)\) for all \(m \in \N\), which follows by Lemma~\ref{lemma: tightness_martingale}. By the same argument we can reduce the analysis of tightness for \((\eta^N, N \in \N)\) in \(\mathcal X\) to the space \( C\big([0,T], H^{-\frac{d}{2}-2-\frac{1}{m} } \big)  \cap L^2\big([0,T], H^{-\frac{d}{2}-2-\frac{1}{m} } \big) \) for big \(m \in \N\). 
    Next, by Lemma~\ref{lemma: tighntess_fluctuation_sequence} and Prokhorov's theorem there exists for every \(\epsilon\) a relative compact set \(K_\epsilon\) such that 
    \begin{equation*}
        \P(\eta^N \not \in K_\epsilon) \le \frac{\epsilon}{2}.
    \end{equation*}
    We define the following set 
    \begin{equation*}
        A_\epsilon:= \Bigl\{ u \in K_\epsilon \colon \int\limits_0^T \norm{u_t}^2_{H^{- \frac{1}{2}(d+\frac{1}{m})}} \Id t   \le M_\epsilon \Bigr\}. 
    \end{equation*}
    Let \((u_n, n \in \N)\) be a sequence in \(A_\epsilon\), then \((u_n, n \in \N)\) is a sequence in \(K_\epsilon\) and, hence, there exists a converging subsequence in \(C\big([0,T], H^{-\frac{d}{2}-2-\frac{1}{m} }\big) \), which we do not rename. 
    Applying~\cite[Proposition 1.52]{Bahouri2011} with some interpolation constant \(\theta\) and \(-\frac{d}{2}-2-\frac{1}{m}<-\frac{d}{2}-\frac{1}{m} < -\frac{1}{2}(d+\frac{1}{m})\) we obtain 
    \begin{align*}
    &\int_0^T \| u_n(t) - u_{n'}(t) \|^2_{H^{-\frac{d}{2}-\frac{1}{m}}(\R^d)} dt \\
    &\quad \leq \left( \int_0^T \| u_n(t) - u_{n'}(t) \|^{2\theta}_{H^{-\frac{1}{2}(d+\frac{1}{m}) }(\R^d)}
    \| u_n(t) - u_{n'}(t) \|^{2(1-\theta)}_{H^{-\frac{d}{2}-2-\frac{1}{m}}(\R^d)} dt \right) \\
    &\quad \leq \left( \int_0^T \| u_n(t) - u_{n'}(t) \|^2_{H^{-\frac{1}{2}(d+\frac{1}{m})}(\R^d)} dt \right)^\theta
    \left( \int_0^T \| u_n(t) - u_{n'}(t) \|^2_{H^{-\frac{d}{2}-2-\frac{1}{m}}(\R^d)} dt \right)^{1-\theta} \\
    &\quad \leq \left( \int_0^T \| u_n(t) - u_{n'}(t) \|^2_{H^{-\frac{1}{2}(d+\frac{1}{m})}(\R^d)} dt \right)^\theta
    \left( T \sup_{t \in [0,T]} \| u_n(t) - u_{n'}(t) \|^2_{H^{-\frac{d}{2}-2-\frac{1}{m}}(\R^d)} \right)^{1-\theta}.
    \end{align*}
The almost sure convergence of \((u_n, n \in \N)\) in \(C\big([0,T], H^{-\frac{d}{2}-2-\frac{1}{m} \big) }\) and the bound provided by the set \(A_\epsilon\) shows that \((u_n, n \in \N)\) is a Cauchy sequence and therefore has convergent subsequence in \(L^2\big([0,T], H^{-\frac{d}{2}-\frac{1}{m} } \big)\). Hence, \(A_\epsilon\) is relative compact in the space \( C\big([0,T], H^{-\frac{d}{2}-2-\frac{1}{m} } \big)  \cap L^2\big([0,T], H^{-\frac{d}{2}-\frac{1}{m} } \big) \) and an application of Chebyshev's inequality yields 
\begin{align*}
    \P(\eta^N \not \in A_\epsilon) \le  \P(\eta^N \not \in K_\epsilon) 
    + \frac{T}{M_\epsilon} \sup\limits_{N \in \N} \sup\limits_{0 \le t \le T }  \E\big(  \norm{\eta_t^N}_{H^{-\frac{1}{2}(d+\frac{1}{m})}(\R^d)}^4 \big)  ,
\end{align*}
which can be made smaller than \(\epsilon\) by choosing \(M_\epsilon\) big enough, since the relative entropy is bounded by Lemma~\ref{lemma: uniform_estimate} and the fact that \(-\tfrac{1}{2}(d+\tfrac{1}{m}) < -\tfrac{d}{2}\). This proves that \((\eta^N, N \in \N)\) is tight in \(\mathcal{X}\). 
    Additionally, the laws of \(W\) and \(\rho\) are tight, since any probability measure on a Polish space is tight. Combining everything proves the claim. 
\end{proof}

A standard consequence of Skorokhod's representation theorem is the following proposition. 
\begin{proposition}\label{prop: skorohod}
There exists a subsequence of \( ((W,\eta^N, \mathcal M^N, \rho), N \in \N) \), which we still denoted by \( ((W,\eta^N, \mathcal M^N, \rho), N \in \N) \) for simplicity, and a new filtered probability space \( (\tilde{\Omega}, \tilde{\cF},  \tilde{\P}) \) with  \(\mathcal W \times \mathcal X \times \mathcal Y \times \mathcal Z \)-valued random variables \( ((\tilde W^N, \tilde{\eta}^N, \tilde{ \mathcal M}^N, \tilde{\rho}^N ),N \in \N)\) and \( (\tilde W, \tilde{\eta}, \tilde{ \mathcal M}, \tilde{\rho} )\) such that:
\begin{enumerate}
\item The sequence of   \( ((\tilde W^N, \tilde{\eta}^N, \tilde{\mathcal M}^N, \tilde{\rho}^N ),N \in \N)\) converges to \( (\tilde{W}, \tilde{\eta}, \tilde{\mathcal M}, \tilde{\rho}) \) in \(\mathcal W \times \mathcal X \times \mathcal Y \times \mathcal Z \) \( \tilde{P} \)-a.s.
    \item For each \( N \in \mathbb{N} \), the law of the \( (\tilde W^N, \tilde{\eta}^N, \tilde{ \mathcal M}^N, \tilde{\rho}^N )\) on \( (\tilde{\Omega}, \tilde{\cF},  \tilde{\P}) \) coincides with the law of \( (W,\eta^N,\mathcal  M^N, \rho) \) on \( (\Omega, \cF, P) \).
\end{enumerate}
\end{proposition}

We need to define a new filtration on the new space \( (\tilde{\Omega}, \tilde{\cF}, \tilde{\P}) \). 
For each \(N \in \N\) we define the new filtration 
\begin{equation*}
    \tilde \cG_t^N = \sigma ( (\tilde W^N_t, \tilde{\eta}^N_t, \tilde{ \mathcal M}^N_t, \tilde{\rho}^N_t )) 
\end{equation*}
and let \(\tilde {\mathcal N}\) be the set of \(\tilde \P \)-null sets in \(\tilde \cF\). Then, set 
\begin{equation*}
    \tilde \cF^N_t = \bigcap\limits_{u > t} \sigma ( \tilde \cG_u^N \cup 
 \tilde {\mathcal N} ) .
\end{equation*}
By the measurability of all processes, we have \(\tilde \cF_t^N \subseteq \tilde \cF\) and therefore it is a filtration. Similar, we define the filtrations \(\tilde \cG\) and \(\tilde \cF\), which correspond to the limiting processes \( (\tilde{W}, \tilde{\eta}, \tilde{ \mathcal M}, \tilde{\rho}) \). 
The question, which arises is \(\tilde W^N\) is still a Brownian motions with respect to \((\tilde \cF_t^N, t \ge 0)\) and \(\tilde W\) is a Brownian motion with respect to \((\tilde \cF_t, t \ge 0)\). This the statement of the next lemma. 
\begin{lemma} \label{lemma: brownian_motion}
    For each \(N\in \N\) the processes \((\tilde W_t^N, t \ge 0)\) and \((\tilde W_t, t \ge 0)\) are Brownian motions on \( (\tilde{\Omega}, \tilde{\cF}, \tilde{\P}) \) with respect to the filtrations \((\tilde \cF_t^N, t \ge 0)\) and \((\tilde \cF_t, t \ge 0)\), respectively. 
\end{lemma}
\begin{proof}
    In the first step, we demonstrate that \((\tilde W_t^N, t \ge 0)\) is a Brownian motion with respect to \((\tilde \cG_t^N, t \ge 0)\). Here, we need to use the fact that all processes \( (W,\eta^N, \mathcal M^N, \rho) \) are initially constructed on the same probability space and adapted. Hence, let \(0 \le s < t \le T\) and 
    \begin{align*}
    \gamma : \mathcal {W}_s \times \mathcal{X}_s \times \mathcal{Y}_s \times \mathcal {Z}_s \mapsto \R 
    \end{align*}
    be a continuous and bounded function, where \(\mathcal W_s, \mathcal{X}_s ,\mathcal Y_s, \mathcal Z_s \) are defined as in~\eqref{eq: spaces} but with the time variable \(s \) instead of \(T\). We obtain 
    \begin{align*}
        &\tilde \E\Big( (\tilde {W}^N_t -\tilde W^N_s) \gamma(\tilde W^N_{|[0,s]},  \tilde \eta^N_{|[0,s]}, \tilde{ \mathcal M}^N _{|[0,s]},\tilde \rho_{|[0,s]}) \Big) \\
        & \quad = \E\Big( ( {W}_t - W_s) \gamma(  W_{|[0,s]},   \eta^N_{|[0,s]}, {\mathcal M}^N _{|[0,s]},\rho_{|[0,s]}) \Big) \\
        &\quad = \E\Big( \E(  {W}_t - W_s |\cF_s) \gamma(  W_{|[0,s]},   \eta^N_{|[0,s]}, {\mathcal M}^N _{|[0,s]},\rho_{|[0,s]}) \Big) \\
        &\quad =0 .
    \end{align*}
    Together with the obvious integrability of \(\tilde W_t\) for all \(t\), we have demonstrated that \(\tilde W^N\) is a continuous martingale with respect to the filtration \((\tilde \cG_t^N , t \ge 0)\). By~\cite[Lemma~67.10]{Rogers1994} the process remains a martingale with respect to the filtration \((\tilde \cF_t^N,t \ge 0)\). Furthermore, applying the same steps, we can show that \(((\tilde W^N_t)^2-t, t \ge 0)\) is a martingale with respect to \((\tilde \cF_t^N,t \ge 0)\).  Hence, applying~\cite[Lemma~72.3]{Rogers1994} we deduce that \(\tilde W^N \) is Brownian motion with respect to \((\tilde \cF_t^N, t \ge 0)\). 
    For the limiting \(\tilde W\) we apply~\cite[p.~526, Proposition ~1.17]{Jacod2003} to find that \(\tilde W\) is a martingale with respect to the filtration \((\tilde \cG_t, t \ge 0)\). Again, by \cite[Lemma~67.10]{Rogers1994} the process remains a continuous martingale with respect to the filtration \((\tilde \cF_t,t \ge 0)\). Since, the almost everywhere convergence of \((\tilde W^N, N \in \N)\) implies the almost everywhere convergence of the squared process, we may utilize the same arguments to demonstrate that \((\tilde W_t^2-t, t \ge 0)\) is a martingale with respect to \((\tilde \cF_t,t \ge 0)\). In this step, we need to incorporate the process \((\tilde{W}_t, t \geq 0)\) as the auxiliary process \(Y\), following the notation in~\cite[p.~526, Proposition~1.17]{Jacod2003}, since, in general, \(\sigma(W^2) \subseteq \sigma(W)\).
    This implies that \((\tilde W_t, t \ge 0)\) is a Brownian motion with respect to \((\tilde \cF_t, t \ge 0)\).  
\end{proof}

\begin{remark} \label{remark: vortex_limit}
Notice that this section relies solely on the uniform estimates established in Section~\ref{sec: uniform_estimates}. Consequently, all statements in this section remain valid in the case \(\nu = 0\), with the interaction kernel given by the Biot--Savart kernel \(k_{\mathrm{vortex}}\).
\end{remark}

\section{Identifying the Limit} \label{sec: identification_limit}
In this section, we demonstrate the convergence of \((\eta^N, N \in \N)\) and, simultaneously, the weak existence of~\eqref{eq: limiting_spde}. After establishing the pathwise uniqueness of the fluctuation SPDE~\eqref{eq: limiting_spde} we prove our Main Theorem~\ref{theorem: main} at the end of the section.  
\subsection{Weak existence of SPDE~\ref{eq: limiting_spde}}
Our first result is regarding the SPDE~\eqref{eq: chaotic_spde}. We will demonstrate that \((\tilde \rho_t, t \ge 0) \), which is the limiting process derived by Skorohod's representation theorem in Proposition~\ref{prop: skorohod}, solves the SPDE~\ref{eq: chaotic_spde} on the new probability space \( (\tilde{\Omega}, \tilde{\cF}, \tilde{\P}) \). 
\begin{lemma} \label{lemma: solution_tilde}
Let \((\tilde \rho_t, t \ge 0) \) be given on  \( (\tilde{\Omega}, \tilde{\cF}, \tilde{\P}) \). Then, \((\tilde \rho_t, t \ge 0) \) solves the SPDE in the sense of Definition~\ref{def: solution_spde}. 
\end{lemma}
\begin{proof}
    Since the probability space \( (\tilde{\Omega}, \tilde{\cF}, \tilde{\P}) \) supports a Brownian motion \((\tilde W_t, t \ge 0)\) with respect to a complete and right continuous filtration \((\tilde \cF_t, t \ge 0)\) we have a unique strong solution of the SPDE by Assumption~\ref{ass: entropy}. Hence, by the generalized Yamada--Watanabe theorem~\cite[Theorem~1.5]{Kurtz2014} there exists a measurable function \(h\) such that \(h(\rho_0,W_\cdot)=\tilde \rho\). 
    Additionally, a simple application of the dominated convergence theorem and the second property of Proposition~\ref{prop: skorohod} demonstrates that the law of \((W,\rho)\) coincides with the law of \((\tilde W, \tilde \rho)\). 
    Hence, since \(h(\rho_0,W_\cdot)=\tilde \rho, \;  \P\)-a.s., we follow \(h(\tilde  \rho_0 , \tilde W) = \tilde \rho, \; \tilde \P\)-a.s. and \((\tilde \rho_t, t \ge 0)\) is a solution of the SPDE~\eqref{eq: chaotic_spde} on the probability space \( (\tilde{\Omega}, \tilde{\cF}, \tilde{\P}) \) with Brownian motion \((\tilde W_t, t\ge 0)\) in the sense of Definition~\ref{def: solution_spde}. 
\end{proof}

In the following lemma, we show that the limiting process \((\tilde{ \mathcal M}_t)_{t \geq 0}\), conditioned on \(\mathcal{F}^{\tilde W}\), is a martingale whose covariance structure is determined by the solution of the stochastic Fokker--Planck equation~\eqref{eq: chaotic_spde}. Our approach follows the general outline of~\cite{Lacker2019}. However, we work on a different probability space, constructed via the Skorohod representation theorem. Additionally, we provide full details, as some steps, though seemingly straightforward, require careful justification. In particular, we establish Lemma~\ref{lemma: quadratic_common_noise}, which appears to be expected but, to the best of our knowledge, lacks an explicit reference in the literature.

\begin{lemma} \label{lemma: characterization_martingale}
    For any \(\varphi \in \testfunctions{\R^d} \), the process \((\tilde {\mathcal M_t}(\varphi), t \ge 0)\) is conditionally on \(\cF^{\tilde W}\) a Gaussian process with covariance structure given by 
    \begin{equation*}
        \tilde \E( \tilde {\mathcal M_t} (\varphi_1) \tilde {\mathcal M_s} (\varphi_2) | \cF^{\tilde W} ) =  \sum\limits_{l=1}^{m} \sum\limits_{q , \hat q =1}^d  \int\limits_0^{\min(s,t)} \langle \sigma_u^{q,l} \sigma_u^{\hat q, l} \partial_{x_q}\varphi \partial_{x_{\hat q }} \phi, \tilde \rho_u \rangle  \Id u 
    \end{equation*}
\end{lemma}
\begin{proof}
    In order to proof the claim we utilize the same characterization as~\cite[Lemma~5.4]{Lacker2019} and demonstrate that \(( \tilde M_t (\varphi), t \ge 0) \) and  
    \begin{equation*}
             \tilde {\mathcal M_t}  (\varphi)  \tilde{ \mathcal M_t}  (\phi) - 
             \sum\limits_{l=1}^{m} \sum\limits_{q , \hat q =1}^d  \int\limits_0^t \langle \sigma_s^{q,l} \sigma_s^{\hat q, l} \partial_{x_q}\varphi \partial_{x_{\hat q }} \phi, \tilde \rho_s \rangle  \Id s , \quad t \ge 0
    \end{equation*}
    conditioned on \(\cF^W\) are martingales. Let \(\gamma_1\) be a real valued continuous function on \(\mathcal W\) and \(\gamma_2\) be a real valued continuous function on \(\mathcal Y_s\). Then, 
    \begin{align*}
        \tilde \E( \gamma_1(\tilde W) \gamma_2(\tilde{ \mathcal M}_{|[0,s]}) (\tilde { \mathcal M_t} -\tilde { \mathcal M_s} )) 
        &= \lim\limits_{N \to \infty} \E( \gamma_1( W) \gamma_2( { \mathcal M}^N_{|[0,s]}) ( { \mathcal M} _t^N- {\mathcal M}_s^N))  \\
        &=0.
    \end{align*}
    Here, the first equality follows form the boundedness of the functions \(\gamma_1, \gamma_2\) and the uniform integrability of \((\tilde {\mathcal M}_t^N(\varphi) , N \in \N)\), which follows from Lemma~\ref{lemma: tightness_common_noise}, the equality in laws and the uniform bound on the relative entropy. The second equality, follows from the fact that \(({\mathcal M}_t^N, t \ge 0)\) is a sum of martingales with respect to initial filtration \((\cF_t, t \ge 0)\), which follows from the strong existence and uniqueness of the interacting particle system~\ref{eq: interacting_particle_system}. 
    For the second martingale, we perform similar steps to find
    \begin{align*}
          &\tilde \E( \gamma_1(\tilde W) \gamma_2(\tilde {\mathcal  M}_{|[0,s]}) (\tilde {\mathcal M}_t (\varphi)  \tilde {\mathcal M}_t (\phi)  - \tilde {\mathcal M}_s (\varphi)  \tilde {\mathcal M}_s (\phi))  )  \\
          &\quad = \lim\limits_{N\to \infty}  \E( \gamma_1( W) \gamma_2( {\mathcal M}^N_{|[0,s]}) ({\mathcal M}_t^N (\varphi) {\mathcal M}_t^N (\phi)  -  { \mathcal M}_s^N (\varphi) { \mathcal M}_s^N (\phi))  ) . 
    \end{align*}
    Now, let us fix \( N \in \N \). Then, 
    \begin{align*}
        &\E( \gamma_1( W) \gamma_2( {\mathcal M}^N_{|[0,s]}) ({ \mathcal M}_t^N (\varphi) { \mathcal M}_t^N (\phi)) \\
        &\quad = \frac{1}{N} \sum\limits_{i,j=1}^N \E\bigg( \gamma_1( W) \gamma_2( { \mathcal M}^N_{|[0,s]})   \int\limits_0^t (\sigma^T(s,X_s^{i}) \nabla \varphi(X_s^{i})) \Id B_s^{i} \int\limits_0^t (\sigma^T(s,X_s^{j}) \nabla \phi(X_s^{j})) \Id B_s^{j}  \bigg) \\
        &\quad = \frac{1}{N} \sum\limits_{i,j=1}^N \sum\limits_{l,\hat l=1}^{m} \sum\limits_{q, \hat q=1}^d  
        \E\bigg( \gamma_1( W) \gamma_2( { \mathcal M}^N_{|[0,s]})\\
        &\quad \quad \cdot \int\limits_0^t  \sigma^{q,l} (s,X_s^{i}) \partial_{x_q} \varphi(X_s^{i})    \Id B_s^{i,l} \int\limits_0^t \sigma^{\hat q,\hat l} (s,X_s^{j}) \partial_{x_{\hat q}} \phi(X_s^{j})  \Id B_s^{j,\hat  l} \bigg). 
    \end{align*}
  The same inequality holds for the time \(s\) and by 
  applying Lemma~\ref{lemma: quadratic_common_noise} we can discard all cross terms. 
  Denoting by 
  \begin{equation*}
      I_t^{i,l,q}(\varphi):= \int\limits_0^t \sigma^{q,l} (u,X_u^{i}) \partial_{x_{q}} \varphi(X_u^{i})  \Id B_u^{i, l}, \quad 
      I_t^{i,l,\hat q}(\phi):= \int\limits_0^t \sigma^{\hat q,l} (u,X_u^{i}) \partial_{x_{\hat q}} \phi(X_u^{i})  \Id B_u^{i, l},
  \end{equation*}
   we obtain 
  \begin{align*}
      &\E( \gamma_1( W) \gamma_2( M^N_{|[0,s]}) ( { \mathcal M}_t^N (\varphi) {\mathcal M}_t^N (\phi)- { \mathcal M}_s^N (\varphi) { \mathcal M}_s^N(\phi)) )\\
      &\quad=  \frac{1}{N} \sum\limits_{i=1}^N \sum\limits_{l=1}^{m} \sum\limits_{q , \hat q =1}^d 
        \E\Big( \gamma_1( W) \gamma_2( {\mathcal M}^N_{|[0,s]}) (I_t^{i,l,q}(\varphi)I_t^{i,l,\hat q}(\phi)- I_s^{i,l,q}(\varphi)I_s^{i,l,\hat q}(\phi) ) \Big) \\
    &\quad = \frac{1}{N} \sum\limits_{i=1}^N \sum\limits_{l=1}^{m}\sum\limits_{q , \hat q =1}^d  
        \E\Big( \gamma_1( W) \gamma_2( {\mathcal M}^N_{|[0,s]}) ((I_t^{i,l,q}(\varphi) - I_s^{i,l,q}(\varphi)) (I_t^{i,l,\hat q}(\phi)-I_s^{i,l, \hat q}(\phi) ) \\
        &\quad \quad +  I_t^{i,l, q}(\varphi)  I_s^{i,l,\hat q}(\phi) +I_s^{i,l, q}(\varphi)  I_t^{i,l, \hat q}(\phi)  - 2 I_s^{i,l,q}(\varphi) I_s^{i,l,\hat q}(\phi) \Big) .
  \end{align*}
Applying Corollary~\ref{cor: quadratic_common_noise} we notice that the sum of the last three stochastic integrals vanish. Finally, applying~\cite[Chapter~3.2, Proposition~2.10]{KaratzasIoannis2009Bmas} to the the filtration \((\sigma (\cF^{W}, \cF_t), t \ge 0)\) we find 
\begin{align*}
     &\E( \gamma_1( W) \gamma_2( { \mathcal M}^N_{|[0,s]}) ({ \mathcal M}_t^N (\varphi) { \mathcal M}_t^N (\phi)- { \mathcal M}_s^N (\varphi){ \mathcal M}_s^N(\phi)) ) \\
     &\quad = \frac{1}{N} \sum\limits_{i=1}^N \sum\limits_{l=1}^{m} \sum\limits_{q , \hat q =1}^d 
        \E\bigg( \gamma_1( W) \gamma_2( { \mathcal M}^N_{|[0,s]})  \\
        & \quad \quad \quad \E\bigg( \int\limits_s^t \sigma^{ q,l} (u,X_u^{i}) 
 \sigma^{\hat q,l}(u,X_u^{i}) \partial_{x_q}\varphi(X_u^{i})\partial_{x_{\hat q}} \phi(X_u^{i}) \Id u    \vert\sigma (\cF^{ W}, \cF_s) \bigg) \bigg) \\
 &\quad = \sum\limits_{l=1}^{m} \sum\limits_{q , \hat q =1}^d 
        \E\bigg( \gamma_1( W) \gamma_2( { \mathcal M}^N_{|[0,s]}) \int\limits_s^t \langle  \sigma^{ q,l}_u
 \sigma^{\hat q,l}_u \partial_{x_q}\varphi \partial_{x_{\hat q}} \phi, \mu_u^N   \rangle \Id u \bigg) \\
 &\quad =   \sum\limits_{l=1}^{m} \sum\limits_{q , \hat q =1}^d 
        \E\bigg( \gamma_1( W) \gamma_2( { \mathcal M}^N_{|[0,s]}) \int\limits_s^t \langle  \sigma^{ q,l}_u
 \sigma^{\hat q,l}_u \partial_{x_q}\varphi \partial_{x_{\hat q}} \phi, \frac{1}{\sqrt N}\eta^N_u + \rho_u    \rangle \Id u \bigg) \\
  &\quad =   \sum\limits_{l=1}^{m} \sum\limits_{q , \hat q =1}^d 
        \tilde \E\bigg( \gamma_1( \tilde W) \gamma_2( \tilde { \mathcal M}^N_{|[0,s]}) \int\limits_s^t \langle  \sigma^{ q,l}_u
 \sigma^{\hat q,l}_u \partial_{x_q}\varphi \partial_{x_{\hat q}} \phi, \tilde \rho_u^N  \rangle \Id u \bigg) \\
 &\quad \quad + \E\bigg( \gamma_1( W) \gamma_2( {\mathcal M}^N_{|[0,s]}) \int\limits_s^t \langle  \sigma^{ q,l}_u
 \sigma^{\hat q,l}_u \partial_{x_q}\varphi \partial_{x_{\hat q}} \phi, \frac{1}{\sqrt N}\eta^N_u  \rangle \Id u \bigg), 
\end{align*}
where we used the defining property of the conditional expectation in the second step. Putting everything together we find 
\begin{align*}
    &\tilde \E\bigg(( \gamma_1(\tilde W) \gamma_2(\tilde {\mathcal M}_{|[0,s]}) \bigg((\tilde {\mathcal M}_t (\varphi)  \tilde { \mathcal M}_t (\phi)  - \tilde { \mathcal M}_s (\varphi)  \tilde { \mathcal M}_s (\phi)) \\
    &\quad \quad - \sum\limits_{l=1}^{m} \sum\limits_{q , \hat q =1}^d \int\limits_s^t \langle  \sigma^{ q,l}_u
 \sigma^{\hat q,l}_u \partial_{x_q}\varphi \partial_{x_{\hat q}} \phi, \tilde \rho_u\rangle \Id u \bigg) \bigg)    \\
 &\quad \le \limsup\limits_{N\to \infty} \bigg| \sum\limits_{l=1}^{m} \sum\limits_{q , \hat q =1}^d 
        \tilde \E\bigg( \gamma_1( \tilde W) \gamma_2( \tilde { \mathcal M}^N_{|[0,s]}) \int\limits_s^t \langle  \sigma^{ q,l}_u
 \sigma^{\hat q,l}_u \partial_{x_q}\varphi \partial_{x_{\hat q}} \phi, \tilde \rho_u^N-\tilde \rho_u  \rangle \Id u \bigg) \bigg| \\
 &\quad \quad + \limsup\limits_{N\to \infty} \bigg|  \sum\limits_{l=1}^{m} \sum\limits_{q , \hat q =1}^d  \E \bigg( \gamma_1( W) \gamma_2( { \mathcal M}^N_{|[0,s]}) \int\limits_s^t \langle  \sigma^{ q,l}_u
 \sigma^{\hat q,l}_u \partial_{x_q}\varphi \partial_{x_{\hat q}} \phi, \frac{1}{\sqrt N}\eta^N_u  \rangle \Id u \bigg) \bigg|. 
\end{align*}
The right hand side vanishes immediately if we can demonstrate that 
\begin{equation} \label{eq: quadratic_common_aux1}
    \sigma^{ q,l}_u
 \sigma^{\hat q,l}_u \partial_{x_q}\varphi \partial_{x_{\hat q}} \phi 
 \in H^{\alpha}(\R^d)
\end{equation}
for \(\alpha > d/2+2\) by the almost everywhere convergence of \((\tilde \rho^N, N \in \N)\) towards \(\tilde \rho\), the uniform bound of \((\eta^N, N \in \N)\) provided by Lemma~\ref{lemma: uniform_estimate} and the boundedness of \(\gamma_1, \gamma_2\). 
But the functions \(\varphi,\phi\) are smooth and by Assumption~\ref{ass: main} the coefficient \(\sigma\) is smooth enough to apply the pointwise multiplication in Theorem~\ref{theorem: holder}. Hence,~\eqref{eq: quadratic_common_aux1} holds and by utilizing the conditional martingale property the covariation formula follows. 
 \end{proof}

In the next Lemma we tackle the non-linear drift term, which is the main difficulty in the limiting procedure. 
\begin{lemma} \label{lemma: drift_convergence}
Let \(\varphi \in \testfunctions{\R^d}\). It holds
\[
    \lim\limits_{N \to \infty} \tilde\E \bigg(\sup\limits_{t \in [0,T] }  \bigg| \int\limits_0^t \langle   \sqrt{N} (\tilde \mu_s^N k*\tilde \mu_s^N -  \tilde{\rho}_s k*\tilde{\rho}_s) -( \tilde{\rho}_s k*\tilde \eta_s + \tilde \eta_s k*\tilde{\rho}_s) , \nabla \varphi \rangle  \Id s  \bigg|   \bigg) =0 ,
    \]
    where \(\tilde \mu_t^N := \frac{1}{\sqrt N} \tilde \eta^N_t+\tilde{\rho}_t \). 
\end{lemma}

\begin{proof}
    We utilize the identity 
    \begin{equation*}
         \sqrt{N}  \big( \tilde \mu_s^N k*\tilde \mu_s^N - \tilde{\rho}_s k*\tilde{\rho}_s \big) 
         = \tilde{\rho}_s k*\tilde \eta_s^N + \tilde \eta_s^N k*\tilde \rho_s  + \frac{1}{\sqrt{N}} \tilde \eta_s^N  k*\tilde \eta_s^N. 
    \end{equation*}
    Hence, for \(\varphi \in \testfunctions{\R^d}\) we obtain 
    \begin{align*}
         &\int\limits_0^t \langle \sqrt{N}( \tilde \mu_s^N k*\tilde \mu_s^N - \tilde{\rho}_s k*\tilde{\rho}_s) -( \tilde{\rho}_s k*\tilde \eta_s + \tilde \eta_s k*\tilde{\rho}_s) , \nabla \varphi \rangle  \Id s \\
        &\quad \le  \bigg( \int\limits_0^T  \frac{\langle \tilde \eta_s^N  k*\tilde \eta_s^N, \nabla \varphi  \rangle}{\sqrt{N}} \Id s 
        +\int\limits_0^T \langle \tilde{\rho}_s k* (\tilde \eta_s^N-\tilde \eta_s), \nabla  \varphi \rangle \Id s 
        + \int\limits_0^T \langle (\tilde \eta_s^N-\tilde \eta_s) k*\tilde{\rho}_s, \nabla  \varphi  \rangle \Id s \bigg) . 
    \end{align*}
    For the first term, we can use Lemma~\ref{lemma: uniform_identification}, the fact that \(\tilde \eta^N\) has the same distribution as \(\eta^N\) and our Assumption~\ref{ass: entropy} on the relative entropy to obtain 
    \begin{align*}
      \frac{1}{\sqrt{N}} \int\limits_0^T \tilde\E\big(| \langle \tilde \eta_s^N  k*\tilde \eta_s^N, \nabla \varphi  \rangle | \big)  \Id s 
        &\le T \sqrt{N} \sup\limits_{0 \le t \le T} \E ( |\langle \nabla \varphi \cdot  k*(\mu_t^N -\rho_t),\mu_t^N -\rho_t \rangle | )  \\
        &\le \frac{CT}{\sqrt{N}} \bigg( \sup\limits_{0 \le t \le T}  \E\bigg( \mathcal{H}(\rho_t^N \vert \rho_t^{\otimes N} \bigg) +1 \bigg) \\
        &\to 0, \; \mathrm{as} \; N \to \infty. 
    \end{align*}
    Utilizing the duality estimate for \(d/2 < \alpha < d/2+1\) for the second and third term, we obtain
    \begin{equation*}
        \langle (\tilde \eta_s^N-\tilde \eta_s) k*\tilde{\rho}_s, \nabla  \varphi  \rangle 
        \le  \norm{\tilde \eta_s^N-\tilde \eta_s}_{H^{-\alpha}(\R^d)} \norm{k*\tilde{\rho}_s \cdot \nabla \varphi}_{H^{\alpha}(\R^d)} 
    \end{equation*}
    and 
    \begin{equation*}
        \langle \tilde{\rho}_s k* (\tilde \eta_s^N-\tilde \eta_s), \nabla  \varphi \rangle 
        \le   \norm{\tilde \eta_s^N-\tilde \eta_s}_{H^{-\alpha}(\R^d)}  \norm{\hat{k}*(\tilde{\rho}_s \nabla \varphi)}_{H^{\alpha}(\R^d)} ,
    \end{equation*}
    where \(\hat{k}(x) = k(-x)\) is the reflection. Applying, first Theorem~\ref{theorem: holder} and then Lemma~\ref{lemma: young} we find 
    \begin{equation*}
        \norm{k*\tilde{\rho}_s \cdot \nabla \varphi}_{H^{\alpha}(\R^d)} 
        \le C \norm{k*\tilde{\rho}_s}_{B^{a}_{\infty,\infty}(\R^d)} \norm{\nabla \varphi}_{H^{\alpha}(\R^d)} 
        \le C  \norm{k}_{L^2(\R^d)} \norm{\tilde{\rho}_s}_{H^{a}(\R^d)} \norm{\nabla \varphi}_{H^{\alpha}(\R^d)} 
    \end{equation*}
    for \(a> \alpha\). Next, we apply first Lemma~\ref{lemma: young} and then Theorem~\ref{theorem: holder} to obtain 
    \begin{equation*}
        \norm{\hat{k}*(\tilde{\rho}_s \nabla \varphi)}_{H^{\alpha}(\R^d)}
        \le \norm{\hat{k}}_{L^2(\R^d)} \norm{\tilde{\rho}_s  \nabla \varphi}_{B_{1,2}^\alpha(\R^d)}
        \le \norm{k}_{L^2(\R^d)} \norm{\tilde{\rho}_s}_{H^a(\R^d)} \norm{\nabla\varphi}_{H^\alpha(\R^d)}. 
    \end{equation*}
Finally, combining everything we arrive at 
\begin{align*}
    &\int\limits_0^T \langle \tilde{\rho}_s k* (\tilde \eta_s^N-\tilde \eta_s), \nabla  \varphi \rangle \Id s 
        + \int\limits_0^T \langle (\tilde \eta_s^N-\tilde \eta_s) k*\tilde{\rho}_s, \nabla  \varphi  \rangle \Id s  \\
        &\quad \le \int\limits_0^T  \norm{\eta_s^N-\tilde \eta_s}_{H^{-\alpha}(\R^d)} \norm{k}_{L^2(\R^d)} \norm{\tilde{\rho}_s}_{H^a(\R^d)} \norm{\nabla \varphi}_{H^{\alpha}(\R^d)} \Id s  ,
\end{align*}
which converges as \(N\to \infty\) by the almost everywhere convergence of \((\tilde \eta^N, N \in \N)\) towards \(\tilde \eta\) in \(C([0,T]; H^{-\alpha}(\R^d))\) and the uniform integrability, which is implied by the uniform square integrability of the \((\tilde \eta^N, N \in \N)\).

\end{proof}

\begin{remark}
   Similar to the situation described in~\cite{nikolaev2024common}, we require \( k \in L^2(\mathbb{R}^d) \) in order to estimate the convolution involving \(\hat{k}\). This mirrors the issue encountered by the author in the stability analysis in~\cite{nikolaev2024common}. 
   Notice that, while some small refinement of the estimation can be achieved by utilizing the fact that \(\eta^N\) is a signed measure, allowing the use of moment estimates to artificially bound the domain \(\mathbb{R}^d\), this approach fails for \(\eta\). The reason is that \(\eta\) is, a priori, only an element of \( H^{-\alpha}(\mathbb{R}^d) \), where discussing moments does not make sense.
\end{remark}

\begin{remark} \label{remark: vortex_stability}
At this stage, our approach for the vortex model encounters a fundamental obstruction in the case of the vortex model. The difficulty arises at the level of the estimate, and consequently in the well posedness, of the term
\[
\norm{k_{\mathrm{vortex}} * (\rho_t \varphi)}_{H^\alpha(\R^2)} .
\]
This estimate corresponds to the critical Hardy--Littlewood--Sobolev inequality in dimension \(d = 2\), and cannot be obtained within our current framework. As a result, it becomes unclear how to give a rigorous meaning to the term \(k_{\mathrm{vortex}} * \eta \rho_t\) on \(\R^2\) in the limiting SPDE~\ref{eq: limiting_spde}, when the fluctuation \(\eta\) is only a distribution in \(H^{-\alpha}(\R^2)\)

At this point, the unboundedness of the domain plays a crucial role, as it prevents the use the fact that \(k_{\mathrm{vortex}} \in L^1(\R^d)\), which would otherwise allow one to control this term. Overcoming this difficulty appears to require new ideas, and we therefore leave this problem for future work.
\end{remark}

\begin{theorem} \label{theorem: weak_existence}
    The limit \(\tilde \eta\) solves the SPDE~\eqref{eq: limiting_spde} in the sense of Definition~\ref{def: limiting_spde} on the stochastic basis \((\tilde \Omega , \tilde \cF, (\tilde \cF_t, t \ge 0), \tilde \P)\).
\end{theorem}
\begin{proof}
The first four points \((1)-(4)\) are direct consequence of Proposition~\ref{prop: skorohod} Lemma~\ref{lemma: brownian_motion}, Lemma~\ref{lemma: solution_tilde}. The fifth property \((5)\) follows by Lemma~\ref{lemma: characterization_martingale}. It remains to demonstrate that the equality in the \((6)\)-th point of Definition~\ref{def: limiting_spde} holds. Here, we follow similar to~\cite{shao2025} the well know approach by Hofmanov\'a and Seidler~\cite{Hofmanova2012}. The idea is classical that instead of using the martingale representation Theorem~\cite[Theorem~3.4.2]{KaratzasIoannis2009Bmas} we can instead demonstrate that 
\begin{equation} \label{eq: theorem_aux1}
\tilde Z_{\cdot}, \quad \tilde Z^2_{\cdot}- \int\limits_0^\cdot |\langle \tilde \eta_s, \nu^{{\mathrm T}} \nabla \varphi \rangle|^2 \Id s , \quad \tilde Z_\cdot \tilde W_\cdot -  \int\limits_0^\cdot \langle \tilde \eta_s, \nu^{{\mathrm T}} \nabla \varphi \rangle \Id s 
\end{equation}
are martingales with respect to the filtration \((\tilde \cF_t, t \ge 0)\), with 
\begin{align*}
    \tilde Z_t :&=   \langle \tilde \eta_t, \varphi \rangle
    -  \langle \tilde \eta_0, \varphi \rangle -    \int\limits_0^t \langle \tilde \rho_s (k*\tilde \eta_s)  , \nabla \varphi \rangle + \langle \tilde \eta_s (k*\tilde \rho_s) , \nabla \varphi  \rangle   \Id t \\
    \quad \quad  & -  \frac{1}{2}\int\limits_0^t  \langle \tilde \eta_s , \mathrm{Tr} \big( (\sigma_s \sigma^{\mathrm{T}}_s + \nu_s \nu^{\mathrm{T}}_s ) \nabla^2 \varphi \big) \rangle \Id s - \tilde {\mathcal M}_t(\varphi)
\end{align*}
and \(\varphi \in \testfunctions{\R^d}\). 
Similar, for \(N\in \N\) we denote the processes \((Z_t^N, t \ge 0)\) by 

\begin{align*}
  Z_t^N  :&=   \langle \eta_t, \varphi \rangle
    -  \langle \eta_0, \varphi \rangle -    \int\limits_0^t \langle  \rho_s (k* \eta_s)  , \nabla \varphi \rangle + \langle  \eta_s (k* \rho_s) , \nabla \varphi  \rangle + \frac{1}{\sqrt N} \langle \eta_s^N k * \eta_s^N, \nabla \varphi \rangle   \Id t \\
    \quad \quad  & -  \frac{1}{2}\int\limits_0^t  \langle  \eta_s , \mathrm{Tr} \big( (\sigma_s \sigma^{\mathrm{T}}_s + \nu_s \nu^{\mathrm{T}}_s ) \nabla^2 \varphi \big) \rangle \Id s - {\mathcal M}_t^N (\varphi). 
\end{align*}

We denote the process \((\tilde Z_t^N, t \ge 0)\) similar to \((Z_t^N,t \ge 0)\) with \(\tilde \eta , \tilde \rho, \tilde{\mathcal M}\) replaced by \(\tilde \eta^N, \tilde \rho^N, \tilde{ \mathcal M^N}\). From the characterization of \((\eta^N_t, t \ge 0)\) is is clear that \((Z_t^N, t \ge 0)\) is a martingale with respect to \((\cF_t, t \ge 0)\), which coincides with \(\hat {\mathcal M}_t^N(\varphi)\). Since \(Z^N\) has the same distribution as \(\tilde Z^N\), we can follow with similar techniques as in Lemma~\ref{lemma: brownian_motion} to demonstrate that \((\tilde Z^N_t, t \ge 0)\) is a martingale with respect to \((\tilde \cG_t, t \ge 0)\). Additionally, we see that 
\begin{equation*}
    (Z^N_\cdot)^2 - \int\limits_0^\cdot |\langle \eta_s, \nu^{{\mathrm T}} \nabla \varphi \rangle|^2 \Id s , \quad  Z_\cdot^N  W_\cdot -  \int\limits_0^\cdot \langle  \eta_s, \nu^{{\mathrm T}} \nabla \varphi \rangle \Id s 
\end{equation*}
are martingales with respect to \((\cF_t,t \ge 0)\) by computing the quadratic variation and covariation of \((\hat{\mathcal M}_t^N, t \ge 0) \). Again, transferring the probability spaces by Proposition~\ref{prop: skorohod} with the technique in Lemma~\ref{lemma: brownian_motion} we find that 
\begin{equation*} 
\tilde Z_{\cdot}^N, \quad  (\tilde Z^N_{\cdot})^2- \int\limits_0^\cdot |\langle \tilde \eta_s^N, \nu^{{\mathrm T}} \nabla \varphi \rangle|^2 \Id s , \quad \tilde Z_{\cdot}^N \tilde W_{\cdot}^N -  \int\limits_0^\cdot \langle \tilde \eta_s^N, \nu^{{\mathrm T}} \nabla \varphi \rangle \Id s 
\end{equation*}
are martingales. See also~\cite[Proposition~4.7]{shao2025} for the technique of Lemma~\ref{lemma: brownian_motion} implemented. Applying~\cite[p.~526, Proposition ~1.17]{Jacod2003}, we find that the processes~\eqref{eq: theorem_aux1} are continuous martingales with respect to \((\tilde \cG_t, t \ge 0)\). Here we used the convergence results provided by Proposition~\ref{prop: skorohod} and Lemma~\ref{lemma: drift_convergence} to conclude that \((\tilde Z^N, N \in N) \) converges to \(\tilde Z\). 
 Utilizing~\cite[Lemma~67.10]{Rogers1994} the martingales~\eqref{eq: theorem_aux1} stay continuous martingales under the augmentation and, therefore, are martingales with respect to \((\tilde \cF_t,t \ge 0)\). 
 Now we are in the position to apply the observation by Hofmanov\'a and Seidler~\cite{Hofmanova2012}, namely that a square integrable martingale is indistinguishable from zero if its quadratic variation vanishes. Using the martingale properties established above, we can compute the quadratic variation of
\begin{equation*}
\tilde Z_\cdot - \int_0^\cdot \langle \tilde \eta_s, \nu_s^{\mathrm T} \nabla \varphi \rangle \, \Id \tilde W_s,
\end{equation*}
and this quadratic variation indeed vanishes.
 
In particular, we find \((\tilde Z_t, t \ge 0)\) is indistinguishable from 
\begin{equation*}
    \bigg( \int\limits_0^t \langle \tilde \eta_s, \nu^{\mathrm T}_s \nabla \varphi\rangle  \Id \tilde W_s, t \ge 0\bigg) . 
\end{equation*}
Consequently, the \((6)\)-th property of Definition~\ref{def: limiting_spde} holds and the Theorem is proven. 
\end{proof}

\subsection{Pathwise uniqueness of limiting SPDE}
In order to prove our main Theorem~\ref{theorem: main} it remains to demonstrate the uniqueness of our limiting SPDE~\eqref{eq: limiting_spde}. 

\begin{theorem}[Uniqueness of the SPDE~\eqref{eq: limiting_spde}] \label{theorem: strong-uniqu}
Given a probability space \((\Omega, \cF, \P)\) and a \(\tilde m\) dimensional Brownian motion \((W_t, t \ge 0)\), a stochastic process \((M_t, t \ge 0)\), a process \((\rho_t ,t \ge 0)\) solving the stochastic Fokker--Planck equation~\eqref{eq: chaotic_spde}. Then, pathwise uniqueness holds in the sense of Definition~\ref{def: pathwise_unique} in the space \(H^{-\alpha}(\R^d)\) for all \(a>  \alpha >  d/2+2\), where \(a\) is the constant in Assumption~\ref{ass: main}. 
\end{theorem}
\begin{proof}
    Let \(\eta^1,\eta^2\) be two solution to the SPDE~\eqref{eq: limiting_spde}. Then, \(\eta = \eta^1-\eta^2\) solves the equation 
    \begin{align*}
        \partial_t \eta =& -  \nabla \cdot (\eta_t (k*\rho_t) ) - \nabla \cdot (\rho_t (k*\eta_t) )   - \nabla \cdot (\eta  \nu^{\mathrm{T}} \Id W_t ) \\
& \quad
+  \frac{1}{2}  \sum\limits_{\alpha,\beta=1}^d \partial_{z_{\alpha}}\partial_{z_{\beta}} \bigg( ( [\sigma_t\sigma_t^{\mathrm{T}}]_{(\alpha,\beta)} + [\nu_t \nu_t^{\mathrm{T}}]_{(\alpha,\beta)} ) \eta_t \bigg)  \Id t 
    \end{align*}
    with initial condition \(\eta_0 =0\). Notice that \(\tau^M\) is indeed a stopping time by Remark~\ref{remark: continuity_of_spde}. 
    Our goal is to apply~\cite[Theorem~5.1]{Krylov1999AnAA} with the following stopping time
    \begin{equation*}
        \tau^M(\omega):= \min \bigg( \inf \{  t \ge 0 \colon \norm{\rho_t}_{H^{a-1}(\R^d)} \ge M \}, T \bigg) 
    \end{equation*}
    for \(a > \alpha+1\). 
    Since, the above SPDE is linear we can just concentrate on the drift part
    \begin{equation*}
        f(t,v,\cdot) = - \nabla \cdot (v (k*\rho_t) ) - \nabla \cdot (\rho_t (k*v) ).
    \end{equation*}
We have 
\begin{align*}
    &\norm{f(t,v,\cdot)-f(t,u,\cdot)}_{H^{-\alpha}(\R^d)}\\
    &\quad \le C  \norm{(v-u)k*\rho_t}_{H^{-\alpha+1}(\R^d)}  + \norm{\rho_t k*(v-u)}_{H^{-\alpha+1}(\R^d)} \\
     &\quad\le C  \norm{v-u}_{H^{-\alpha+1}(\R^d)}  \norm{k}_{L^2(\R^d)} \norm{\rho_t}_{H^{a-1}(\R^d)}  + \norm{k*(v-u)}_{B_{\infty,2}^{-\alpha+1}(\R^d)} \norm{\rho_t}_{H^{a-1}(\R^d)} \\
     &\quad\le C  \norm{\rho_t}_{H^{a-1}(\R^d)} \norm{k}_{L^2(\R^d)}  \norm{v-u}_{H^{-\alpha+1}(\R^d)} \\
     &\quad\le C M \norm{k}_{L^2(\R^d)}  \norm{v-u}_{H^{-\alpha+1}(\R^d)}, 
\end{align*}
where we used the multiplication inequality in Theorem~\ref{theorem: holder} and the fact that \(\alpha < a \). 
This confirms the main assumption in~\cite[Theorem~5.1]{Krylov1999AnAA}. All other assumptions of~\cite[Theorem~5.1]{Krylov1999AnAA} readily follow from the assumptions on the coefficients. Consequently, there exists a unique solution in the space \(L^2_{\cF^{W}}([0,T];H^{-\alpha+2}(\R^d))\), which is continuous in the space \(H^{-\alpha +1}(\R^d)\), i.e. \(C([0,T],H^{-\alpha+1}(\R^d))\). However, the trivial solution also solves the SPDE~\eqref{eq: limiting_spde} and we obtain \[\P\Big(\sup\limits_{0 \le t \le \tau^M  } \norm{\eta_t}_{H^{-\alpha+1}(\R^d)} = 0\Big) = 1\] on the time interval \([0,\tau^M]\) for all \(M \in \N\). But the map \(t \mapsto \norm{\rho_t}_{H^{a-1}(\R^d)}\) is continuous by Remark~\ref{remark: continuity_of_spde}. Therefore, \(\tau^M \to T, \; \P\)-a.e. and the theorem is proven. 
\end{proof}

\begin{proof}[Proof of Theorem~\ref{theorem: main}]
    By Theorem~\ref{theorem: weak_existence} there exists a weak solution of the limiting fluctuation SPDE~\eqref{eq: limiting_spde}. By Theorem~\ref{theorem: strong-uniqu} strong uniqueness holds for the fluctuation SPDE~\eqref{eq: limiting_spde}. Applying the general Yamada--Watanabe Theorem~\cite[Theorem~1.5]{Kurtz2014} we conclude the strong well-posedness of SPDE~\eqref{eq: limiting_spde}. In particular this implies uniqueness in law. Hence, by Lemma~\ref{lemma: tighntess_fluctuation_sequence} we can find a subsequence of \((\eta^N, N \in \N)\), which converges to the unique solution of the SPDE~\eqref{eq: limiting_spde}. Since this holds for every subsequence and the law of the limiting point is unique~\cite[Theorem~1.5]{Kurtz2014} we obtain that the whole sequence \((\eta^N, N \in \N)\) converges in law to the fluctuation SPDE~\eqref{eq: limiting_spde}. 
\end{proof}

\begin{appendices}

\section{Besov spaces}\label{secA1}

We recall the multiplication inequalities for Besov spaces~\cite[Corollary~1 and Corollary~2]{Weber2017} and~\cite[Theorem 2.8.2]{Triebel1978}.
\begin{theorem}[H\"older's inequality for Triebel--Lizorkin spaces] \label{theorem: holder}
Let \(s_1 > 0 > s_2\), \(p,q,p_1,p_2 \in [1, \infty]\) such that 
\begin{equation*}
    \frac{1}{p} = \frac{1}{p_1}+\frac{1}{p_2}.
\end{equation*}
Then, the map \((f,g) \mapsto fg\) extends to a continuous linear map from \(B_{p_1,q}^{s_1}(\R^d) \times B_{p_2,q}^{s_1}(\R^d) \) to \(B_{p,q}^{s_1}(\R^d)\) and 
\begin{equation*}
    \norm{fg}_{B_{p,q}^{s_1}(\R^d)} \le C \norm{f}_{B_{p_1,q}^{s_1}(\R^d)}
    \norm{g}_{B_{p_2,q}^{s_1}(\R^d)}.
\end{equation*}
If, in addition \(s_1+s_2 > 0\), then the map \((f,g) \mapsto fg\) extends to a continuous linear map from \(B_{p_1,q}^{s_1}(\R^d) \times B_{p_2,q}^{s_2}(\R^d) \) to \(B_{p,q}^{s_2}(\R^d)\) and 
\begin{equation*}
    \norm{fg}_{B_{p,q}^{s_2}(\R^d)} \le C \norm{f}_{B_{p_1,q}^{s_1}(\R^d)}
    \norm{g}_{B_{p_2,q}^{s_2}(\R^d)}.
\end{equation*}
 Moreover, for \(s \in \R\), \(0 < p <\infty\), \(0 < q \le \infty\) and \(a>\max (s,\frac{d}{p}-s) \), we have the inequality 
\begin{equation*}
    \norm{f g}_{B^{s}_{p,q}(\R^d)} \le \norm{f}_{B^{a}_{\infty,\infty}(\R^d)} \norm{g}_{B^{s}_{p,q}(\R^d)}. 
\end{equation*}
\end{theorem}
Next, we also require Young's inequality~{\cite[Theorem 2.]{Schilling2022}}. 

\begin{lemma}[Young's inequality for Besov spaces] \label{lemma: young}
Let \( s \in \mathbb{R} \), \( q,q_1 \in (0, \infty] \), and \( p, p_1, p_2 \in [1, \infty] \) be such that:
\[
1 + \frac{1}{p} = \frac{1}{p_1} + \frac{1}{p_2} \quad \mathrm{and} \quad \frac{1}{q} \le \frac{1}{q_1} + \frac{1}{2}.
\]
If \( f \in B^{s}_{p_1,q}(\R^d) \) and \( g \in L_{p_2}(\R^d) \), then \( f * g \in B^{s}_{p,q}(\R^d) \) and
\[
\|f * g\|_{B^{s}_{p,q}(\R^d)} \leq C \|f\|_{B^{s}_{p_1,q_1}(\R^d)} \cdot \|g\|_{L^{p_2}(\R^d)},
\]
where \( C > 0 \) is a constant independent of \( f \) and \( g \). 
\end{lemma}

\section{Auxiliary Lemmata}\label{secB}
In this section we proof Lemma~\ref{lemma: versions_of_stochastic_integrals}, Lemma~\ref{lemma: modification_super}, a version of a conditional cancellation property.  
Let us recall~\cite[Lemma~5]{Flandoli2005}. 
\begin{lemma} \label{lem: appendix_aux}
    Let \( \varphi \mapsto S(\varphi) \) be a linear continuous mapping from a separable Banach space \( E \) to \( L^0(\Omega) \) (with the convergence in probability). Assume that there exists a random variable \( C(\omega) \) such that for all \( \varphi \in E \), we have  
\[
|S(\varphi)(\omega)| \leq C(\omega) \|\varphi\|_E \quad \text{for } \mathbb{P}\text{-a.e. } \omega \in \Omega.
\]  
Then there exists a measurable mapping \( \omega \mapsto S(\omega) \) from \( (\Omega, \cF, \mathbb{P}) \) to the dual \( E' \) such that for all \( \varphi \in E \), we have  
\[
\langle S(\omega), \varphi \rangle = S(\varphi)(\omega),
\]  
hence \( S(\omega) \) is a pathwise realization of \( S(\varphi) \).

\end{lemma}
\begin{proof}[Proof of Lemma~\ref{lemma: versions_of_stochastic_integrals}]

For any \(j,\tilde l\) we have 
\begin{equation*}
    \E \bigg( \int\limits_0^t \norm{ \nu^{j,\tilde{l}}(s,\cdot)  \eta_s^N}_{H^{-(\alpha-1)}(\R^d)}^2 \Id s \bigg)   
    \le C  \E \bigg( \int\limits_0^t \norm{ \eta_s^N}_{H^{-(\alpha-1)}(\R^d)}^2 \Id s \bigg) < \infty,
\end{equation*}
where we utilized Assumption~\ref{ass: main} and Lemma~\ref{lemma: uniform_estimate} with the fact that \(\alpha - 1 > d/2\). 
Hence, we can interchange the integral and the linear map to obtain
    \begin{align*}
       | \hat{\mathcal M}_t^N(\varphi)|
       &=  \bigg|\sum\limits_{j=1}^d \sum\limits_{\tilde{l}=1}^{\tilde{m}} \int\limits_0^t   \langle \nu^{j,\tilde{l}}(s,\cdot)  \partial_{z_j} \varphi(\cdot), \eta_s^N \rangle  \Id W_s^{\tilde{l}} \bigg| \\
       &= \bigg|\sum\limits_{j=1}^d \sum\limits_{\tilde{l}=1}^{\tilde{m}} \bigg \langle \int\limits_0^t   \nu^{j,\tilde{l}}(s,\cdot)   \eta_s^N  \Id W_s^{\tilde{l}}  , \partial_{z_j} \varphi(\cdot) \bigg\rangle \bigg| \\
       &\le \sum\limits_{j=1}^d \sum\limits_{\tilde{l}=1}^{\tilde{m}} \norm{\int\limits_0^t   \nu^{j,\tilde{l}}(s,\cdot)   \eta_s^N  \Id W_s^{\tilde{l}}}_{H^{-(\alpha-1)}(\R^d)} \norm{\varphi}_{H^\alpha(\R^d)}.
    \end{align*}
    By the first inequality and the Burkholder-Davis-Gundy (BDG) inequality for Hilbert space-valued martingales~\cite[Theorem~1.1]{Marinelli2016}, we conclude that the random variable is integrable. Hence, Lemma~\ref{lem: appendix_aux} is applicable. 
    For the integral \(M_t^N\) we use~\cite[Lemma~8]{Flandoli2005} and Parseval's identity to obtain 
    \begin{align*}
        &\bigg|\frac{1}{\sqrt{N}} \sum\limits_{i=1}^N \int\limits_0^t  ( \sigma^{\mathrm{T}}(s,X_s^{i}) \nabla \varphi(X_s^{i}) ) \Id B_s^{i}\bigg| \\
        &\quad \le \frac{1}{\sqrt{N}}  \sum\limits_{i=1}^N  \sum\limits_{j=1}^d \sum\limits_{l=1}^m  \bigg|\int\limits_0^t   \sigma^{j,l}(s,X_s^{i}) \partial_{x_j} \varphi(X_s^{i}) ) \Id B_s^{i,l}\bigg| \\
        &\quad = \frac{1}{\sqrt{N}}  \sum\limits_{i=1}^N  \sum\limits_{j=1}^d \sum\limits_{l=1}^m \int_{\R^d} 
        \mathcal{F}(  \varphi(z))\int\limits_0^t \sigma^{j,l}(s,X_s^{i})  \exp(i z \cdot X_s^{i} )  \Id B_s^{i,l} \Id z \\
        &\quad \le C(N) \sup\limits_{i,j,l} \bigg(\int\limits_{\R^d}\bigg| \frac{
        \mathcal{F}(\partial_{x_j} \varphi(z)))}{(1+|x|^2)^{-\frac{(\alpha-1)}{2}}} \bigg|^2 \Id z  \bigg)^{\frac{1}{2}}  \\
        &\quad\quad \cdot \bigg( \int_{\R^d} \frac{1}{(1+|z|^2)^{\alpha-1}} \bigg|\int\limits_0^t \sigma^{j,l}(s,X_s^{i})  \exp(i z \cdot X_s^{i} ) \Id B_s^{i,l} \bigg|^2 \Id z\bigg)^{\frac{1}{2}} \\
        &\quad \le C(N) \norm{\varphi}_{H^{\alpha}(\R^d)} \sup\limits_{i,l}
        \bigg( \int_{\R^d} \frac{1}{(1+|z|^2)^{\alpha-1}} \bigg|\int\limits_0^t  \sigma^{j,l}(s,X_s^{i})  \exp(i z \cdot X_s^{i} ) \Id B_s^{i,l} \bigg|^2 \Id z\bigg)^{\frac{1}{2}} .
    \end{align*}
    Again, using the fact that \(\alpha-1 > d/2\) it is an easy exercise to show that the random variable is integrable by utilizing It\^{o}'s isometry. Hence, the claim follows by Lemma~\ref{lem: appendix_aux}.
    The fact that the versions are progressively measurable, follows immediately by the Pettis theorem and the fact that stochastic integrals are continuous. 
\end{proof}

The proof of Lemma~\ref{lemma: modification_super} is based on~\cite[Lemma~2.3]{Xianliang2023}. For the sake of completeness, we reproduce the proof here once again with a different exponent.

  \begin{proof}[Proof of Lemma~\ref{lemma: modification_super}]
We start with Taylor's expansion:
\[
\int_{\R^{dN}} \rho^{\otimes N} \exp\Bigl( N \Bigl|\bigl\langle \phi, \mu_N\otimes\mu_N \bigr\rangle\Bigr|^{4} \Bigr) \Id x
=\sum_{m=0}^\infty \frac{1}{m!}\int_{\mathbb{R}^{dN} } \bar{\rho}_N \Bigl( N \Bigl|\bigl\langle \phi, \mu_N\otimes\mu_N \bigr\rangle\Bigr|^{4}\Bigr)^m \Id x.
\]
For the $m$-th term, using 
\[
\mu_N = \frac{1}{N}\sum_{i=1}^N \delta_{x_i},
\]
we can write
\begin{align*}
&\frac{1}{m!}\int_{\R^{dN}} \rho^{\otimes N} \Bigl( N \Bigl|\bigl\langle \phi, \mu_N\otimes\mu_N \bigr\rangle\Bigr|^{4} \Bigr)^m \Id x \\
&\quad = \frac{1}{m!} N^m \int_{\R^{dN}}  \rho^{\otimes N}\Bigl(\frac{1}{N^2}\sum_{i,j=1}^N \phi(x_i,x_j)\Bigr)^{4m} \Id x \\
&\quad =\frac{1}{m!} N^{m-8m} \sum_{i_1,\ldots,i_{2qm},j_1,\ldots,j_{4m}=1}^N \int_{\R^{dN}} \rho^{\otimes N} \prod_{\kappa =1}^{4m} \phi(x_{i_\kappa},x_{j_\kappa}) \Id x,
\end{align*}

We now split the analysis into two cases.

\medskip

\textbf{Case 1.} If $8m > N$, then using the uniform bound $\|\phi\|_{L^\infty}$ we have
\begin{align*}
&\frac{ N^{m-8m}}{m!} \sum_{i_1,\ldots,i_{4m},j_1,\ldots,j_{4m}=1}^N  \int_{\R^{dN}} \rho^{\otimes N}\prod_{\kappa =1}^{4m} \phi(x_{i_\kappa},x_{j_\kappa}) \Id X_N  \\
&\quad \le \frac{N^{-7m}}{m!} (N^2)^{4m}\,\|\phi\|_{L^\infty}^{4m}
\le \frac{N^m}{m!} \|\phi\|_{L^\infty}^{4m}
\le m^{-1/2} e^m \frac{N^m}{m^m}  \|\phi\|_{L^\infty}^{4m} \\
&\quad \le  m^{-1/2} e^m 8^{m} \|\phi\|_{L^\infty(\R^d) \times \R^d }^{4m}
\end{align*}
where we utilized Stirling's formula 
\begin{equation*}
    x! = c_x \sqrt{2\pi x } \bigg( \frac{x}{e} \bigg)^x 
\end{equation*}
for \(1 < c_x < \frac{11}{10}\). 
At the end of the proof we require the smallness condition on the \(L^\infty\)-norm of \(\phi\) such that 
\begin{equation*}
    \beta_0^m:=   e^m 8^{m} \|\phi\|_{L^\infty(\R^d) \times \R^d }^{4m} < 1 
\end{equation*}
\medskip

\textbf{Case 2.} If \(8 \le 8m\le N\) 
We estimate the \( m \)-th term by counting how many choices of multi indices \((i_1, \dots, i_{4m}, j_1,\ldots, j_{4m})\) lead to a non-vanishing integral. If there exists a pair \((i_q, j_{q})\) such that 
\begin{equation*}
i_q \neq j_{q} \; \mathrm{and} \; i_q, j_q \notin \{i_\kappa, j_\kappa \} \; \mathrm{for \; any} \;   \kappa \neq q, 
\end{equation*} 
then the variables \( x_{i_q} \) and \( x_{j_q} \) appear exactly once in the integration and the cancellation property gets activated. 

We introduce the following notation:
\begin{itemize}
    \item Let \(l\) denote the number of \(x_{i_\kappa}\) or \(x_{j_\kappa}\), which appear exactly once in the integral. 
    \item Let $p$ denote the number of \(x_{i_\kappa}\) or \(x_{j_\kappa}\), which appear at least twice in the integral.  
\end{itemize}
A crucial observation is that if \(l > 4m \), there must exists a pair \((i_q,j_q)\), which only appears once. Indeed, since we have \(8m \) indices and they appear in pairs there must exists a pair which only appears once. Then by Fubini's theorem the cancellation property applies, see~\cite[page~101]{Xianliang2023}.   
Additionally, by the condition in the sum, we have the following relations: 
\begin{equation*}
    8\le 8m \le N, \quad 0 \le l \le  4 m , \quad 1 \le p \le \frac{8m-l}{2}. 
\end{equation*}

For fixed $l$ and $p$, there are $\binom{N}{l}\binom{N-l}{p}$ choices of indices. Moreover, once these indices are chosen, there are 
\[
\binom{4m}{l} 2^l \, l! \, p^{8m-l}
\]
ways to arrange them so that the cancellation condition is respected.
We find 
\begin{align*}
    &\frac{1}{m!} N^{-7m} \sum_{i_1,\ldots,i_{2qm},j_1,\ldots,j_{2qm}=1}^N \int_{\R^{dN}} \rho^{\otimes N} \prod_{\kappa =1}^{2qm} \phi(x_{i_\kappa},x_{j_\kappa}) \Id x \\
    &\quad \le 
    \frac{1}{m!} N^{-7m} \norm{\phi}_{L^\infty(\R^d \times \R^d )}^{4m}\sum_{l=0}^{4m} \sum_{p=1}^{4m-l/2} \binom{N}{l}\binom{N-l}{p} \binom{4m}{l} 2^l \, l! \, p^{8m-l} \\
    &\quad \le \norm{\phi}_{L^\infty(\R^d \times \R^d)}^{4m}\sum_{l=0}^{4m} \sum_{p=1}^{4m-l/2} \frac{N! N^{-7m}}{(N-p-l)!} \frac{1}{m!p!} \binom{4m}{l} 2^l  \, p^{8m-l}.
\end{align*}
Applying Stirling's formula to \(m!\) and \(p!\) we obtain 
\begin{equation*}
    \frac{N! N^{-7m}}{(N-p-l)!} \frac{1}{m!p!} \binom{4m}{l} 2^l  \, p^{8m-l}
    \le N^{p+l-7m} 2^{l} e^{m+p}   \binom{4m}{l} \frac{p^{8m-l-p}}{m^m}.
\end{equation*}
Moreover, we have
\begin{equation*}
    \frac{\binom{4m}{l}p^m}{m^m} \le \frac{2^{4m} (4m-l/2)^m }{m^m}
    \le 2^{6m}  
\end{equation*}
Plugging it into the previous inequality we arrive at 
\begin{equation*}
    \frac{N! N^{-7m}}{(N-p-l)!} \frac{1}{m!p!} \binom{4m}{l} 2^l  \, p^{8m-l}
    \le  (N/p)^{p+l-7m}   e^{m+p} 2^{l+5m}.
\end{equation*}
Now, \(p \le  N\) and \(p+l-7m \le \tfrac{8m-l}{2}+l-7m \le 4m+l/2 -7m \le 0 \). This implies
\begin{equation*}
    \sum_{l=0}^{4m} \sum_{p=1}^{4m-l/2} \frac{N! N^{-7m}}{(N-p-l)!} \frac{1}{m!p!} \binom{4m}{l} 2^l  \, p^{8m-l}
    \le 4m (4m-l/2) e^{p+6m}\le e^{14m}. 
\end{equation*}

A careful combinatorial analysis shows that the $m$-th term is bounded by
\[
\|\phi\|_{L^\infty(\R^d \times \R^d)}^{4m}\,e^{14m},
\]
i.e. it is controlled by $\alpha_0^m$ with 
\[
\alpha_0:= e^{14}\,\|\phi\|_{L^\infty(\R^d \times \R^d )}^4<1.
\]
Choosing, the \(L^\infty\)-norm of \(\phi\) small enough guarantees the last strict inequality. 

\medskip

Combining the estimates from both cases, we obtain
\[
\int_{ \mathbb{R}^{dN} } \bar{\rho}_N \exp\Bigl( N \Bigl|\bigl\langle \phi,\mu_N\otimes\mu_N \bigr\rangle\Bigr|^2\Bigr) \Id x \le 1+\sum_{m=1}^{\lfloor N/8\rfloor} \alpha_0^m + \sum_{m=\lfloor N/8\rfloor+1}^\infty \beta_0^m.
\]
Since
\[
\sum_{m=1}^\infty \alpha_0^m = \frac{\alpha_0}{1-\alpha_0}\quad \text{and}\quad \sum_{m=1}^\infty \beta_0^m = \frac{\beta_0}{1-\beta_0},
\]
the desired bound follows:
\[
\int_{\mathbb{T}^{dN}} \bar{\rho}_N \exp\Bigl( N \Bigl|\bigl\langle \phi,\mu_N\otimes\mu_N \bigr\rangle\Bigr|^4\Bigr) \Id  x  \le 1+\frac{\alpha_0}{1-\alpha_0} + \frac{\beta_0}{1-\beta_0}.
\]
\end{proof}

 \begin{lemma} \label{lemma: quadratic_common_noise}
 Let \((\Omega,\cF,(\cF_t, t\ge 0),\P)\) be filtered probability space and \((\cF_t, t\ge 0)\) satisfies the usual conditions with two independent Brownian motions \((B_t^1, t\ge 0)\), \((B_t^2, t\ge 0)\) with respect to the filtration \((\cF_t, t\ge 0)\). 
 Let \(f,g \colon \Omega \times [0,T] \to \R \) be measurable, adapted and 
 \begin{equation*}
     \E\bigg(\int\limits_0^T |f_t|^2+|g_t|^2 \Id t \bigg)< \infty. 
 \end{equation*}
 Let \(\cG\) be sub-\(\sigma\)-algebra of \(\cF\)  such that \(\cG\) is independent of the Brownian motion \(B^1\), \(B^2\) and \(\sigma(\cF_t,\cG) = \sigma(\tilde \cF_t, \cG) \) for some \(\sigma\)-algebra \(\tilde \cF_t\) such that \(\cG\) is independent of \(\tilde \cF_t\) and the sigma algebras generated by the Brownian motions \(\sigma(B^1_s, s \ge t)\) and \(\sigma(B_s^2, s \ge t)\). Let \(h_1\) be a bounded \(\cG\)-measurable function and \(h_2\) be a bounded \(\cF_s\) measurable function.  
 Then 
 \begin{equation*}
     \E\bigg( h_1 h_2 \int\limits_0^t  f(s,\cdot) \Id B_s^{1} \int\limits_0^t  g(s,\cdot) \Id B_s^{2} \bigg)
     = \E\bigg( h_1 h_2 \int\limits_0^s  f(s,\cdot) \Id B_s^{1} \int\limits_0^s  g(s,\cdot) \Id B_s^{2} \bigg) . 
 \end{equation*}
 \end{lemma}
 
\begin{remark}
The independence assumption on \(\mathcal{G}\) may appear unusual at first. 
To see why it is natural, consider the simple case in which 
\(\mathcal{F}_t = \sigma(B^1_s, B^2_s, W_s : s \le t)\) for three independent Brownian motions 
\(B^1, B^2, W\), and let \(\mathcal{G} = \sigma(W)\) 
(we ignore the usual augmentation for clarity). 
Then
\(
\sigma(\mathcal{G}, \mathcal{F}_t) 
= \sigma\big( (B^1_s, B^2_s : s \le t),\, W \big),
\)
and the independence required in Lemma~\ref{lemma: quadratic_common_noise} is satisfied. 
In the setting of the main theorem the sigma algebra is slightly larger, as it also contains the usual augmentation and the independent initial condition. 
However these additions remain independent of the future increments of the Brownian motions, so the assumption in the Lemma continues to hold.
\end{remark}

\begin{proof}
We start by decomposing the integral.     
\begin{align*}
     \E\bigg( h_1 h_2 \int\limits_0^t  f (s,\cdot) \Id B_s^{1} \int\limits_0^t  g (s,\cdot) \Id B_s^{2} \bigg) 
     & = \E\bigg( h_1 h_2  \int\limits_0^s  f (s,\cdot) \Id B_s^{1} \int\limits_0^s  g (s,\cdot) \Id B_s^{2} \bigg) \\
     &\quad \quad+ \E\bigg( h_1 h_2  \int\limits_s^t   f(s,\cdot) \Id B_s^{1} \int\limits_0^s  g (s,\cdot) \Id B_s^{2} \bigg) \\
     &\quad \quad+ \E\bigg(h_1 h_2  \int\limits_0^s   f(s,\cdot) \Id B_s^{1} \int\limits_s^t   g (s,\cdot) \Id B_s^{2} \bigg) \\
     &\quad \quad +E\bigg( h_1 h_2  \int\limits_s^t   f(s,\cdot) \Id B_s^{1} \int\limits_s^t  g (s,\cdot) \Id B_s^{2} \bigg)
    \end{align*}
The first term on the right-hand side is exactly what we need. It remains to explain why the last three terms vanish. 
We resort to an approximation argument.  
Assume \((f_n,n \in \N)\) and \((g_n, n \in \N)\) are two sequence of simple functions, where for each \(n \in \N\) the functions are given by 
\begin{align*}
    f_n(\omega,t) &= a_0(\omega) \indicator{\{s\}}(t) + \sum\limits_{i=1}^{m} a_i(\omega) \indicator{(t_i,t_{i+1}]}(t)  \\
    g_n(\omega,t) &= b_0(\omega) \indicator{\{s\}}(t) + \sum\limits_{i=1}^{\hat m} b_i(\omega) \indicator{(\hat t_i, \hat t_{i+1}]}(t) , 
\end{align*}
with some real numbers \((t_i, i =1, \ldots, m)\), \((\hat t_i, i =1, \ldots, \hat m)\) with \(t_i=s,t_{m}=T,\hat t_0 = s, \hat t_{\hat m}=T\), some \(\cF_{t_i}\)-measurable random variables \(a_i\) and some \(\cF_{\hat t_i}\)-measurable random variables \(b_i\), which are square integrable. 
For the second term we find 
\begin{align} \label{eq: appendix_aux1}
\begin{split}
     &\E\bigg( h_1 h_2 \int\limits_s^t   f(s,\cdot) \Id B_s^{1} \int\limits_0^s  g(s,\cdot) \Id B_s^{2} \bigg) \\
     &\quad = \E \bigg(h_1h_2 \int\limits_0^s  g(s,\cdot) \Id B_s^{2} \E\bigg( \int\limits_s^t   f(s,\cdot) \Id B_s^{1} \vert \sigma(\cF_s,\cG)  \bigg) \bigg) \\
     &\quad = 0. 
     \end{split}
\end{align}
The last inequality follows by the fact that for the approximation \((f_n, n \in \N)\) we find 
\begin{align*}
    \E\bigg( \int\limits_s^t   f_n(s,\cdot) \Id B_s^{1} \vert \sigma(\cF_s,\cG)  \bigg)
    &= \sum\limits_{i=1}^m \E(  a_i \E(  (B_{t_{i+1}}^1-B_{t_i}^1)  \vert \sigma(\cF_{t_i},\cG) )  \vert \sigma(\cF_s,\cG)  ) \\
    &= 0 ,
\end{align*}
where we used \(t_i \ge s\) in the first step and the independency of the Brownian motion of the sigma algebra \(\sigma(\cF_{t_i},\cG)\) in the last step. Additionally, the conditional expectation is a \(L^2\) contraction and, therefore, 
\begin{align*}
& \E\bigg( \bigg|  \E\bigg( \int\limits_s^t   f(s,\cdot) \Id B_s^{1} \vert \sigma(\cF_s,\cG) \bigg) \bigg|^2 \bigg)\\
     &\quad = \E\bigg( \bigg| \E\bigg( \int\limits_s^t   f_n(s,\cdot) \Id B_s^{1} \vert \sigma(\cF_s,\cG)\bigg)  - \E\bigg( \int\limits_s^t   f(s,\cdot) \Id B_s^{1} \vert \sigma(\cF_s,\cG) \bigg) \bigg|^2 \bigg) \\
     &\quad \le \E \bigg( \bigg| \int\limits_s^t   f_n(s,\cdot) \Id B_s^{1}  - \int\limits_s^t   f(s,\cdot) \Id B_s^{1} \bigg|^2 \bigg)  \\
     &\quad \xrightarrow[]{n \to \infty} 0 
\end{align*}
by the properties of stochastic integrals, see~\cite[Chapter~3.2]{KaratzasIoannis2009Bmas}. This establishes~\eqref{eq: appendix_aux1}. 
The third follows, analogously. The last term vanishes again by a simple approximation procedure. 

    \begin{equation*}
    E\bigg( h_1 h_2  \int\limits_s^t   f(s,\cdot) \Id B_s^{1} \int\limits_s^t  g (s,\cdot) \Id B_s^{2} \bigg)
          = \sum\limits_{i=1}^n \sum\limits_{j=1}^m \E\bigg( h_1 h_2 a_{i}(B_{t_{i+1}}^1-B_{t_{i}}^1) b_{j}(B_{\hat t_{j+1}}^2-B_{\hat t_{j}}^2)  \bigg)
    \end{equation*}
Let us assume \(t_{i} \le t_{i+1} \le \hat t_j \le \hat t_{j+1} \). Then conditioning the above on \(\cF_{\hat t_{j+1}}\) and using the martingale property of the Brownian motion with respect to \((\cF_t,t \ge 0)\) the term vanishes. The same thing hold if the roles of \(i,j\) are switched. Let us assume \(t_{i} \le \hat t_j  \le \min(\hat t_{j+1},t_{i+1}), \). We find 
\begin{align*}
    &\E\bigg( \gamma_1( W) \gamma_2( M^N_{|[0,s]}) a_{i}(B_{t_{i+1}}^1-B_{t_{i}}^1) b_{j}(B_{\hat t_{j+1}}^2-B_{\hat t_{j}}^2)  \bigg) \\
    &\quad =\E\bigg( \gamma_1( W) \gamma_2( M^N_{|[0,s]}) a_{i}  b_{j} \E( (B_{t_{i+1}}^1-B_{\hat t_j}^1 +B_{\hat t_j}^1 -B_{t_{i}}^1) (B_{\hat t_{j+1}}^2-B_{ \hat t_{j}}^2) | \cF_{\hat t_{j}} )  \bigg) \\
    &\quad = \E\bigg( \gamma_1( W) \gamma_2( M^N_{|[0,s]}) a_{i}  b_{j} (B_{\hat t_j}^1 -B_{t_{i}}^1)  \E(  (B_{\hat t_{j+1}}^2-B_{ \hat t_{j}}^2) | \cF_{\hat t_{j}} )  \bigg) \\
    &\quad =0, 
\end{align*}
where we used the fact that \((B_{t_{i+1}}^1-B_{\hat t_j}^1) (B_{\hat t_{j+1}}^2-B_{ \hat t_{j}}^2) \) is independent of \(\cF_{\hat t_j}\) in the second step and the martingale property in the last step. Again, interchanging the roles of \(i,j\) we have proven all cases of indices and and the Lemma is proven. 
\end{proof}
Notice, that we only required the independence for the last term corresponding to the time intervals \([s,t],[s,t]\). Choosing a function \(g\), which gets cut-off at time \(s\), we obtain the following corollary. 
\begin{corollary} \label{cor: quadratic_common_noise}
    In the situation of Lemma~\ref{lemma: quadratic_common_noise} we obtain 
    \begin{equation*}
         \E\bigg( h_1 h_2 \int\limits_0^t  f(s,\cdot) \Id B_s^{1} \int\limits_0^s  g(s,\cdot) \Id B_s^{1} \bigg)
     = \E\bigg( h_1 h_2 \int\limits_0^s  f(s,\cdot) \Id B_s^{1} \int\limits_0^s  g(s,\cdot) \Id B_s^{1} \bigg) . 
    \end{equation*}
\end{corollary}

\end{appendices}


\medskip

{\noindent\textbf{Acknowledgements}: The author has received funding from the project MeCoGa, Mean Field Control and Games, at the University of Padova through the STARS@UNIPD - NextGenerationEU program, as well as from the European Union's Horizon Europe research and innovation programme under the Marie Sk\l;odowska-Curie Actions Staff Exchanges (Grant Agreement No. 101183168, Call: 
HORIZON-MSCA-2023-SE-01) and from the DFG CRC/TRR 388 “Rough Analysis, Stochastic Dynamics and Related Fields”, Project A11.}
\vspace{2em}

{\noindent\textbf{Disclaimer}: Funded by the European Union. Views and opinions expressed are however those of the author only and do not necessarily reflect those of the European Union or the European Education and Culture Executive Agency (EACEA). Neither the European Union nor EACEA can be held responsible for them.}

\bibliography{quellen}
\bibliographystyle{amsalpha}
\end{document}